\documentclass[11pt, a4paper]{article}

\usepackage{amsmath, amsfonts,amscd,flafter,epsf,amssymb,epsfig,subfigure,bm,url, mathrsfs, amsthm, comment}
\usepackage[margin=.9in]{geometry}

\allowdisplaybreaks

\usepackage{mathtools}

\usepackage{pgfplots,tikz}
\usepackage{xkeyval,tkz-base,tikz-3dplot}
\usepackage{tqft}
\usetikzlibrary{arrows, calc, shapes, automata, backgrounds, petri, positioning,fadings,decorations.pathreplacing, spy, intersections}
\usetikzlibrary{external}
\usetikzlibrary{fillbetween}
\usetikzlibrary{decorations.markings}
\usetikzlibrary{lindenmayersystems}
\tikzset{
	external/system call={
	xelatex \tikzexternalcheckshellescape
	-halt-on-error -interaction=batchmode --shell-escape --enable-write18
	-jobname "\image" "\texsource"}
}
\usepgfplotslibrary{external} 
\def\nextAngle{0}
\tikzset{
	next angle/.style={
		in=#1+180,
		out=\nextAngle,
		prefix after command= {\pgfextra{\def\nextAngle{#1}}}
	},
	start angle/.style={
		out=#1,
		nangle=#1,
	},
	nangle/.code={
		\def\nextAngle{#1}
	}
}

\usepackage{etoolbox}
\AtBeginEnvironment{tikzcd}{\tikzexternaldisable}
\AtEndEnvironment{tikzcd}{\tikzexternalenable}

\newtheorem{thm}{Theorem}[section]
\newtheorem{lemma}[thm]{Lemma}
\newtheorem{prop}[thm]{Proposition}
\newtheorem{cor}[thm]{Corollary}
\newtheorem{conj}[thm]{Conjecture}

\theoremstyle{definition}

\theoremstyle{remark}

\newcommand{\bbh}{\mathbb{H}}

\newcommand{\cald}{\mathcal{D}}
\newcommand{\caln}{\mathcal{N}}
\newcommand{\calp}{\mathcal{P}}
\newcommand{\calx}{\mathcal{X}}
\newcommand{\caly}{\mathcal{Y}}

\makeatletter \newcommand\dd[1]{\ensuremath{%
  \mathop{}\!\mathrm{d}#1\@ifnextchar\dd{\!}{}}}
\makeatother

\makeatletter
\DeclareFontFamily{U}{tipa}{}
\DeclareFontShape{U}{tipa}{m}{n}{<->tipa10}{}
\newcommand{\arc@char}{{\usefont{U}{tipa}{m}{n}\symbol{62}}}%

\newcommand{\arc}[1]{\mathpalette\arc@arc{#1}}

\newcommand{\arc@arc}[2]{%
  \sbox0{$\m@th#1#2$}%
  \vbox{
    \hbox{\resizebox{\wd0}{\height}{\arc@char}}
    \nointerlineskip
    \box0
  }%
}
\makeatother

\newcommand{\abs}[1]{\left| #1 \right|}
\newcommand{\gauss}[1]{\left[ #1 \right]}
\newcommand{\set}[1]{\left\{ #1 \right\}}
\newcommand{\bracket}[1]{\left( #1 \right)}

\author{
	Wujie Shen\\
	\texttt{shenwj22@mails.tsinghua.edu.cn}
}

\title{Nonsimple closed geodesics with given intersection number on hyperbolic surfaces}

\pgfplotsset{compat=1.18} 

\begin{document}

\maketitle

\begin{abstract}
We prove that the minimal length of a closed geodesic with self-intersection number $k$ on any finite-type hyperbolic surface is $2\cosh^{-1}(1+2k)$ for $k>1750$. This improves the previously known threshold $k > 10^{13350}$ established in \cite{Y2801}. Our proof is independent of the methods in \cite{Y2801}.
\end{abstract}

\section{Introduction}

The study of nonsimple closed geodesics on hyperbolic surfaces plays a fundamental role in two-dimensional hyperbolic geometry, spectral theory, and Teichmüller theory. A natural question arises: For a hyperbolic surface, let $M_k$ denote the minimal length among all closed geodesics with exactly $k$ self-intersections. Does $M_k \to \infty$ as $k \to \infty$, and if so, what is the precise asymptotic behavior of $L_k$?

There has been extensive work on this problem. Hempel \cite{H2797} first established a universal lower bound $2\log(1+\sqrt{2})$ for nonsimple closed geodesics, which Yamada \cite{Y2799} later improved to the sharp bound $4\log(1+\sqrt{2}) = 2\cosh^{-1}(3)$, proving it is attained on ideal pairs of pants. Basmajian \cite{B2794} proved that the lengths of nonsimple geodesics grow arbitrarily large with the self-intersection number (\cite[Corollary 1.2]{B2794}). For the specific case of hyperbolic pairs of pants, Baribaud \cite{B2795} computed exact minimal lengths for geodesics with prescribed self-intersection numbers.

Let $\omega$ be either a closed geodesic or a geodesic segment on a hyperbolic surface, with $\ell(\omega)$ denoting its length and $|{\omega\cap\omega}|$ its self-intersection number. Here $|{\omega\cap\omega}|$ counts transverse self-intersections with multiplicity, where each intersection point having $n$ preimages contributes $\binom{n}{2}$ to the total count.

For a fixed hyperbolic surface, Basmajian \cite{B2803} proved that any $k$-geodesic (a closed geodesic with exactly $k$ self-intersections) has length at least $C\sqrt{k}$, where $C > 0$ is a constant depending only on the hyperbolic structure. Later, Hanh Vo \cite{Y2805} established the exact minimal length of $k$-geodesics for all sufficiently large $k$ (depending on the surface) in the case of hyperbolic surfaces with at least one cusp.

For a hyperbolic surface $X$, define $I(k,X)$ as the minimal self-intersection number among all shortest closed geodesics with at least $k$ self-intersections. By definition, $I(k,X) \geq k$. Erlandsson and Parlier \cite{EP2802} proved that $I(k,X)$ is bounded above by a function depending only on $k$ that exhibits linear growth as $k \to \infty$. However, to the best of our knowledge, no hyperbolic surface $X$ is known to satisfy $I(k,X) = k$ for all $k \geqslant 1$.

Let $M_k$ be the infimum of lengths of geodesics of self-intersection number at least $k$ among all finite-type hyperbolic surfaces, i.e. metric complete hyperbolic surfaces without boundary, and have finite number of genuses and cusps. Basmajian showed (\cite[Corollary 1.4]{B2803}) that
\begin{equation}\label{eqn:growth_bas}
\tfrac12\log \frac{k}2\leqslant M_k \leqslant 2\cosh^{-1}(2k+1)
\end{equation}
He also showed that $M_k$ is realized by a $k$-geodesic on some hyperbolic surface. 

\begin{conj}
When $k\geqslant1$,
\begin{equation}\label{eqn:length_asymptotic}
M_k=2\cosh^{-1}(1+2k)=2\log(1+2k+2\sqrt{k^2+k})
\end{equation}
and the equality holds when $\Gamma$ is a corkscrew geodesic(See definition below) on a thrice-punctured sphere.
\end{conj}

In \cite[Theorem 1.1]{Y2800} Shen-Wang improved the lower bound of $M_k$, that $M_k$ has explicit growth rate $2\log{k}$, and for a closed geodesic of length $L$, the self intersection number is no more than $9L^2e^{\frac{L}{2}}$. The exact value for $M_k$ for sufficiently large $k$ is computed in \cite[Theorem 1.1]{Y2801}:

\begin{thm}\label{thm:main00}
 Conjecture 1.1 holds when $k>10^{13350}$. 
\end{thm}

In \cite{Y2801}, the authors first noticed that when the length of a $k$-geodesic is smallest, it must lies on a cusped hyperbolic surface. And following the main result of \cite{Y2805} to finish the proof.

In the present paper we give a different proof and a better result: 
\begin{thm}\label{thm:main}
 Conjecture 1.1 holds when $k>1750$. 
\end{thm}
We begin by applying the thick-thin decomposition to the surface (Section~\ref{section2}). The uniform lower bound on the injectivity radius in the thick part enables precise control of self-intersection numbers, which we develop in Section~\ref{section3}. To analyze the thin parts, we adapt the methodology of \cite[Lemma 2.5]{Y2805}, yielding an exact count of self-intersections (Section~\ref{section4}). Combining these results, we conclude the proof in Section~\ref{section5}.

Our results represent a significant improvement, as the bound $1750$ is substantially smaller than the previous estimate of $10^{13350}$. Moreover, this work opens two new possibilities: first, computer-assisted verification of Conjecture 1.1 becomes feasible for all $k \leq 1750$; second, it suggests a potential pathway to prove that $I(k,X) = k$ holds for all $k \geq 1$ on the thrice-punctured sphere.

Theorem \ref{thm:main} can be generalized to general orientable finite-type hyperbolic surfaces, possibly with geodesic boundaries, since they can be doubled to get a surface as in Theorem \ref{thm:main}.

\subsubsection*{Acknowledgement}

The author would like to thank Professor Jiajun Wang for introducing me this question, Professor Shing-Tung Yau for some helpful discussions, and Yuhao Xue for his interest in the problem and pointing mistakes in the paper.

\section{Neighborhoods of sufficiently short geodesics and cusps}\label{section2}

In this section, we establish a thick-thin decomposition for hyperbolic surfaces that may include cusps, following an approach similar to \cite{Y2800}. Since the injectivity radius admits a universal lower bound on the thick part, we can effectively bound the self-intersection number in this region.

Let $L\geqslant4\log(1+\sqrt{2})>3.5$ be a constant. Let $\Sigma$ be an oriented, metrically complete hyperbolic surface of finite type without boundary. Topologically $\Sigma$ is an orientable surface of genus $g$ with $n$ punctures such that $2g+n\geqslant 3$. Denote the length of a curve $c$ on $\Sigma$ by $\ell(c)$. 

Let $\Gamma$ be a closed geodesic with length $L=\ell(\Gamma)\geqslant4\log(1+\sqrt{2})>3.5$ on $\Sigma$. Suppose $\Gamma$ is represented as a local isometry $f:S^1\to\Sigma$, where $S^1$ is a circle with length $L$. Let $\cald\subset\Sigma$ be the set of self-intersection points of $\Gamma$, that is,
$$\cald=\set{x\in\Sigma\,:\, \exists\,s,t\in S^1, f(s)=f(t)=x, s\neq t}$$
The \emph{self-intersection number} of $\Gamma$ is defined as
\begin{equation}\label{eqn:self_intersection_number_defn}
\abs{\Gamma\cap\Gamma}:=\sum_{x\in \cald}\binom{\# f^{-1}(x)}{2}
\end{equation}

\subsection{A thick-thin decomposition}

Similar as \cite{Y2800} we define the collection $\calx=\{c_1,...,c_d\}$ of simple closed geodesics of length less than 1 and we have 

\begin{lemma}
The geodesics in $\calx$ are pairwisely disjoint. 
\end{lemma}

\begin{proof}
If $c_i,c_j\in\calx$ with $c_i\cap c_j\neq\emptyset$, the collar lemma \cite[Lemma 13.6]{FM2012} implies that the collar
$$N(c_i):=\set{x\in\Sigma:d(x,c_i)<\sinh^{-1}\bracket{\frac1{\sinh(\ell(c_i)/2)}}}$$
is an embedded annulus. Suppose $y\in c_i\cap c_j$, for all $y'\in c_j$ we have
$$d(y',c_i)\leqslant d(y',y)\leqslant \frac{\ell(c_j)}2<\sinh^{-1}(1)<\sinh^{-1}\left(\frac1{\sinh(\ell(c_i)/2)}\right)$$
hence $y'\in N(c_i)$, so $c_j\subseteq N(c_i)$, since $N(c_i)$ is an annulus, the only simple closed geodesic in $N(c_i)$ is $c_i$ itself, a contradiction.

\end{proof}

For each $1\leqslant{i}\leqslant{d}$ we define neighborhood
\begin{equation}\begin{aligned}
N_3(c_i)=&\set{x\in{\Sigma}:d(x,c_i)<{\log{\frac{4}{\ell(c_i)}}}}
\end{aligned}\end{equation}
Since for all $t>0$
$$\sinh^{-1}\bracket{\frac1{\sinh(t/2)}}\geqslant\log\frac4{t}$$
we have $N_3(c_i)\subseteq N(c_i)$(\cite[Lemma 2.1]{Y2800}). Additively we define
\begin{equation}\begin{aligned}
N_0(c_i)=&\set{x\in{\Sigma}:d(x,c_i)<{\log{\frac{4}{\ell(c_i)}}-\log2}}
\end{aligned}\end{equation}
$N_0(c_i),N_3(c_i)$ both an annulus, and $N_0(c_i)\subseteq N_3(c_i)$.

We also defined the set $\caly$ of punctures of $\Sigma$ in \cite{Y2800}.

In the upper half-plane model for $\bbh^2$, let $\Gamma_0$ be a cyclic group generated by a parabolic isometry of $\bbh^2$ fixing the point $\infty$, assume $\Gamma_0(-1,1)=(1,1)$ in $\bbh^2$. Let $H_c=\set{(x,y)\in\bbh^2\,\big|\,y\geqslant c}$ be a horoball. Each cusp can be modelled as $H_c/\Gamma_0$ for some $c$ up to isometry, and is diffeomorphic to $S^1\times[c,\infty)$ so that each circle $S^1\times\set{t}$ with $t\geqslant c$ is the image of a horocycle under $p$. Each circle is also called a \emph{horocycle} by abuse of notation. The circle $S^1\times\set{t}$ with $t\geqslant c$ is called an \emph{Euclidean circle}. A cusp is \emph{maximal} if it lifts to a union of horocycles with disjoint interiors such that there exists at least one point of tangency between different horocycles.

 Each puncture $c_i$ of $\Sigma$ has a maximal cusp whose boundary Euclidean circle $c$ has 
$\ell(c)\geqslant4$(Adams, \cite{A2817}), and the cusp of area $4$ that can be lifted to $\bbh^2$.
The projection $p$ maps the the triangle $\infty PQ$ to the cusp of area $4$ at $c_i$ and $p$ maps the interior of the triangle homeomorphically. Choose points $A_0, A_3$ on the ray from $P$ to $\infty$, and $B_0, B_3$ on the ray from $Q$ to $\infty$ so that
$$d(P,A_3)=d(Q,B_3)=\log2,\qquad d(A_3,A_0)=d(B_3,B_0)=\log2$$

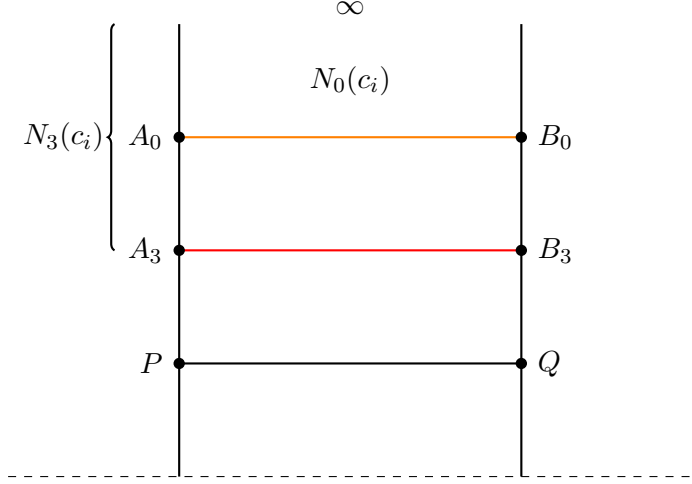
\begin{figure}[htbp]
\begin{center}
\tikzexternaldisable
\begin{tikzpicture}
\def\x{1.5}
\draw[thick] (-1.5*\x,0) to (-1.5*\x,4*\x);
\draw[thick] (1.5*\x,0) to (1.5*\x,4*\x);
\draw[thick] (-1.5*\x, \x) to (1.5*\x,\x);
\draw[thick, color=red] (-1.5*\x, 2*\x) to (1.5*\x,2*\x);
\draw[thick, color=orange] (-1.5*\x, 3*\x) to (1.5*\x,3*\x);
\draw[dashed] (-3*\x,0) to (3*\x,0);
\path (-1.5*\x,\x) node[circle, fill, inner sep=1.5pt, label=left:{$P$}]{};
\path (1.5*\x,\x)  node[circle, fill, inner sep=1.5pt, label=right:{$Q$}]{};

\path (-1.5*\x,2*\x) node[circle, fill, inner sep=1.5pt, label=left:{$A_3$}]{};
\path (1.5*\x,2*\x)  node[circle, fill, inner sep=1.5pt, label=right:{$B_3$}]{};

\path (-1.5*\x,3*\x) node[circle, fill, inner sep=1.5pt, label=left:{$A_0$}]{};
\path (1.5*\x,3*\x)  node[circle, fill, inner sep=1.5pt, label=right:{$B_0$}]{};
\draw [thick, decoration={
	brace, raise=0.1cm
	}, decorate
] (-2*\x, 2*\x) -- (-2*\x, 4*\x) 
	node [pos=0.5,anchor=east,xshift=-0.1cm] {$N_3(c_i)$}; 
\draw (0,3.5*\x) node {$N_0(c_i)$};
\draw (0,4*\x) node[above] {$\infty$};	
\end{tikzpicture}
\end{center}
\tikzexternalenable
\caption{\label{fig:nbhd_cusps}
Neighborhoods of cusps.}
\end{figure}

(Similar as \cite[Section 2.2]{Y2800})

For $j=0,3$, let $N_j(c_i)$ be the image of the triangle $\infty A_jB_j$ under $p$, $N(c_i)$ is the image of $\infty PQ$.

The following theorem is a generalization of the collar lemma, wherever the collar lemma in the compact case is proved in \cite[Lemma 13.6]{FM2012}:

\begin{lemma}
  For distinct $c_i,c_j\in\calx\cup\caly$, $N_3(c_i)\cap N_3(c_j)=\emptyset$.
\end{lemma}

\begin{proof}
 For all $c_i,c_j\in\calx\cup\caly$ $c_i$ and $c_j$ are not homotopy equivalent, so we choose a pants decomposition that for each $c_i\in\calx\cup\caly$, $c_i$ is a boundary(or infinity boundary) of a pair of pants. Only to prove that if $c_i$, $c_j$ are boundary components of same pair of pants $P$, then $N_3(c_i)\cap N_3(c_j)\neq\emptyset$ in $P$.

Suppose the boundary components of $P$ are $c_i,c_j,c_k$. When $c_i,c_j,c_k$ are all short geodesics, the proof is completed in \cite{FM2012}. So we only need to prove the case when one of $c_i,c_j,c_k$ is a cusp, without loss of generality assume $c_i$ is a cusp and we prove $N_3(c_i), N_3(c_j), N_3(c_k)$ pairwisely disjoint in $P$. $P$ can be constructed by gluing 2 hexagons along 3 nonadjacent boundary segments. In the upper half plane model of $\bbh^2$, let $P_1(-1,0)$, $P_2(1,0)$, $\ell_1$ and $\ell_2$ are the lines $x=-1$ and $x=1$. Assume $A_1(-1,a_j)$, $B_1(1,a_k)$, $\gamma_j$ and $\gamma_k$ are the half circle centered at $P_1, P_2$ and of radius $a_j,a_k$(in Euclidean coordinate). $A_2B_2$ is the geodesic orthogonal to half circle $\gamma_j$ and $\gamma_k$, $\widetilde{c_j}\subseteq\gamma_j$ and $\widetilde{c_k}\subseteq\gamma_k$ are the arcs $A_1A_2$ and $B_1B_2$. Note that if $c_j$ is a cusp, then $a_j=0$ and $A_1=A_2=P_1$, similarly when $c_k$ is a cusp. Then  the segment $\ell_1$ from $\infty$ to $A_1$, $\widetilde{c_j}$ ,geodesic arc $A_2B_2$, $\widetilde{c_k}$,and the segment $\ell_2$ from $B_1$ to $\infty$ are boundary components of an ideal hexagon $P'$, $P$ is constructed by gluing two copies of $P'$ along $\infty A_1$, $A_2B_2$, $B_1\infty$.

Assume $C_1(-1,2)$, $C_2(1,2)$, then $\partial N_3(c_i)$ is by gluing 2 copies of Euclidean segment $C_1C_2\subseteq P'$, and $N_3(c_i)$ is by gluing 2 copies of the region between $l_1,l_2$ and above $C_1C_2$. 
\begin{enumerate}
\item
If $c_j\in\calx$ is a short geodesic, let $Q_1(-b_j,0),Q_2(b_k,0)$ be the endpoints of the geodesic containing $A_2B_2$, we have $-1\leqslant b_j<b_k\leqslant1$. $\widetilde{c_j}$ perpendicular to $A_2B_2$ implies that
$$\bracket{1+\frac{b_k-b_j}2}^2=\bracket{\frac{b_k+b_j}2}^2+a_j^2$$ 
Assume $Q$ is the midpoint of segment $Q_1Q_2$ and $\theta:=\angle{P_1QA_2}$, we have
\begin{align*}
\frac{\ell(c_j)}2=d(A_1,A_2)=\log\tan\bracket{{\frac{\pi}4+\frac{\theta}2}}=\log\bracket{\frac{2a_j+\sqrt{4a_j^2+(b_k+b_j)^2}}{b_k+b_j}}
\end{align*}
Hence we have
\begin{align*}
\sinh{d(A_1,C_1)}\sinh{d(A_1,A_2)}&\geqslant \sinh\log\frac2{a_j}\sinh\log\bracket{\frac{2a_j+\sqrt{4a_j^2+(b_k+b_j)^2}}{b_k+b_j}}\\
&=\bracket{\frac1{a_j}-\frac{a_j}4}\cdot\frac{2a_j}{b_k+b_j}=1+\frac{(3-b_j)(1-b_k)}{2(b_k+b_j)}\geqslant1
\end{align*}
Hence $d(A_1,C_1)\geqslant w(\ell(c_j))$. Hence $N_3(c_i)\cap N_3(c_j)\cap P=\emptyset$.

\item
If $c_j$ is a cusp, then $Q_1(-1,0)$. Let $O'$ be the midpoint of $C_1P_1$ and $C_1'$ is the intersection of $A_2B_2$ and circle $C$ $(x+1)^2+(y-1)^2=1$. By gluing 2 copies of $P'$, 2 copies of the arc $C_1C_1'$ of $C$ is a horocycle of $c_j$. Since the vertical coordinate of $C_1'$ is no more than 1, let $C_1''(0,1)$ is on the arc $C_1C_1'$, hence we have
$$\ell(C_1C_1')\geqslant\ell(C_1C_1'')\geqslant\int^{\frac{\pi}2}_{0}\frac1{1+\cos x}dx=1$$

 hence the length of the horocycle is no less than 2, hence $N_3(c_i)\cap N_3(c_j)=\emptyset$ in $P$, similarly $N_3(c_i)\cap N_3(c_k)=\emptyset$ in $P$.

\item
Next we prove $N_3(c_j)\cap N_3(c_k)=\emptyset$. When one of $c_j$, $c_k$ is cusp we already proved above, we only need to consider the case $c_j,c_k$ are both short geodesics. Let $l$ be the line $x=\frac{b_k-b_j}2$, we only need to prove: the half $N_3(c_j)\cap P'$ of $N_3(c_j)\cap P$ lies in the left half $\{(x,y)\in\bbh^2:x<\frac{b_k-b_j}2,y>0\}$ of $\bbh^2$. Only to prove $d(A_1A_2,l)\geqslant w(\ell(c_j))$. In fact,
\begin{align*}
 d(A_1A_2,l)=d(A_2,l)&=-\log\tan\frac{\theta}2=\log\bracket{\frac{b_k+b_j+\sqrt{4a_j^2+(b_k+b_j)^2}}{2a_j}}
\end{align*}
Hence we have
\begin{align*}
 \sinh{d(A_1A_2,l)}\sinh{\frac{\ell(c_i)}2}&=\frac{b_k+b_j}{2a_j}\frac{2a_j}{b_k+b_j}=1
\end{align*}

Hence $N(c_j)\cap N(c_k)=\emptyset$, we proved the lemma.

\end{enumerate}

\end{proof}

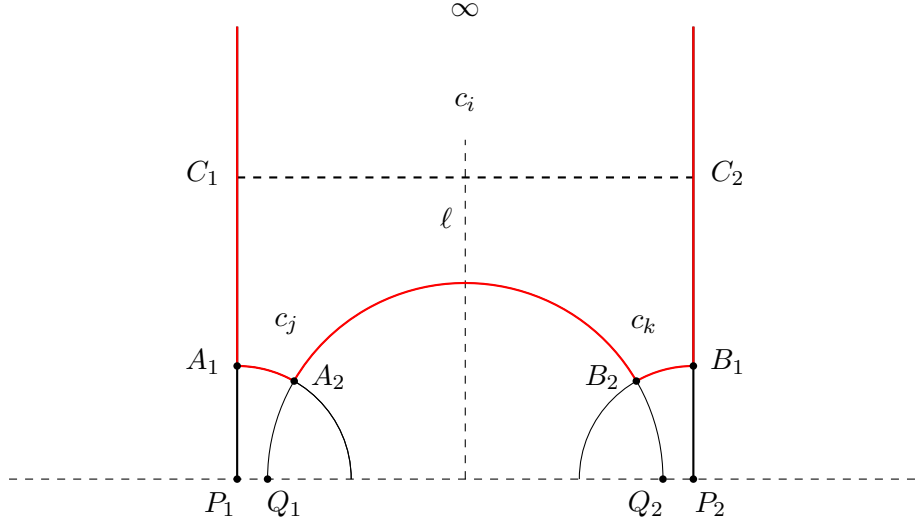
\begin{figure}[htbp]
\begin{center}
\tikzexternaldisable
\begin{tikzpicture}
\def\x{2}

\draw[thick] (-1.5*\x,0) to (-1.5*\x,3*\x);
\draw[thick] (1.5*\x,0) to (1.5*\x,3*\x);

\draw[thick, dashed] (-1.5*\x, 2*\x) to (1.5*\x,2*\x);
\draw[thin, dashed] (0,0) to (0,4.5);

\draw[dashed] (-3*\x,0) to (3*\x,0);

\draw[thin] (3,1.5) arc (90:180:1.5); 
\draw[thin] (-3,1.5) arc (90:0:1.5); 
\draw[thin] (-3,1.5) arc (90:0:1.5);  

\draw[thin] (2.6,0) arc (0:180:2.6); 

\draw[thick, color=red] (2.25,1.3) arc (30:150:2.6); 
\draw[thick, color=red] (3,1.5) arc (90:120:1.5); 
\draw[thick, color=red] (-3,1.5) arc (90:60:1.5); 
\draw[thick, color=red] (-1.5*\x,1.5) to (-1.5*\x,3*\x);
\draw[thick, color=red] (1.5*\x,1.5) to (1.5*\x,3*\x);

\draw (0,2.5*\x) node {$c_i$};

\draw (0,3*\x) node[above] {$\infty$};

\path (-1.5*\x,4) node[circle, inner sep=1pt, label={[shift={(180:0.45)}, anchor=center]{$C_1$}}]{};
\path (1.5*\x,4) node[circle, inner sep=1pt, label={[shift={(0:0.45)}, anchor=center]{$C_2$}}]{};
\path (-1.5*\x,0) node[circle, fill, inner sep=1pt, label={[shift={(240:0.45)}, anchor=center]{$P_1$}}]{};
\path (1.5*\x,0) node[circle, fill, inner sep=1pt, label={[shift={(300:0.45)}, anchor=center]{$P_2$}}]{};
\path (-2.6,0) node[circle, fill, inner sep=1pt, label={[shift={(300:0.45)}, anchor=center]{$Q_1$}}]{};
\path (2.6,0) node[circle, fill, inner sep=1pt, label={[shift={(240:0.45)}, anchor=center]{$Q_2$}}]{};
\path (-3,1.5) node[circle, fill, inner sep=1pt, label={[shift={(180:0.45)}, anchor=center]{$A_1$}}]{};
\path (1.5*\x,1.5) node[circle, fill, inner sep=1pt, label={[shift={(0:0.45)}, anchor=center]{$B_1$}}]{};
\path (2.25,1.3) node[circle, fill, inner sep=1pt, label={[shift={(180:0.45)}, anchor=center]{$B_2$}}]{};
\path (-2.25,1.3) node[circle, fill, inner sep=1pt, label={[shift={(0:0.45)}, anchor=center]{$A_2$}}]{};
\path (2.8,2) node[circle, inner sep=1pt, label={[shift={(180:0.45)}, anchor=center]{$c_k$}}]{};
\path (-2.8,2) node[circle, inner sep=1pt, label={[shift={(0:0.45)}, anchor=center]{$c_j$}}]{};
\path (0.2,3.4) node[circle, inner sep=1pt, label={[shift={(180:0.45)}, anchor=center]{$\ell$}}]{};

	  			
\end{tikzpicture}
\end{center}
\tikzexternalenable
\caption{\label{fig:nbhd_cusps2}
A hexagon of $P$}
\end{figure}

Let $\caln_t:=\bigcup_{c\in\calx\cup\caly}N_0(c)$ be the thin part and $\caln_T:=\Sigma\setminus\caln_t$ be the thick part.

\subsection{Injective radius estimate}

The injective radius of each point in $\caln_T$ has a universal lower bound:

\begin{lemma}\label{lemma:inj_radius_outside_nbhds}
  For all $x\in\caln_T$, the injective radius of $x$ is no less than 
$$\log\bracket{\frac{1+\sqrt5}2}>0.48$$
\end{lemma}

\begin{proof}
  
We prove the lemma by contradiction. Let $x$ be a point in the thick part $\Sigma\setminus\caln_t$ and suppose that the injectivity radius $r_0$ at $x$ satisfies $r_0<\log\bracket{\frac{1+\sqrt5}2}$.

There exists a homotopically nontrivial simple closed curve $\gamma$ through $x$ with $\ell(\gamma)=2r_0<1$. $\gamma$ is freely homotopic to either a (unique) simple closed geodesic $\gamma^\prime$, or a puncture $c$ of $\Sigma$.

Suppose that $\gamma$ is freely homotopic to a puncture $c_i$ of $\Sigma$. Let $\widetilde{x}_1$ and $\widetilde{x}_2$ be lifts of $x$ as in Lemma 2.5. Then $d(\widetilde{x}_1,\widetilde{x}_2)=2r_0<0.97$ and hence $x\in N_3(c_i)$. Now $x\notin N_0(c_i)$ and $d(x,\partial N_3(c_i))\leqslant\log2$. By Lemma 2.5, we have $\sinh(r_0)\geqslant\frac12$, hence $r_0\geqslant\log\bracket{\frac{1+\sqrt5}2}$, contradiction.

Now suppose that $\gamma$ is freely homotopic to a simple closed geodesic $\gamma^\prime$. Then $\ell(\gamma^\prime)\leqslant 2r_0<0.97$ and hence $\gamma^\prime$ is a curve $c_i$ in $\calx$. There are two cases according to whether $c_i\cap \gamma=\emptyset$.

\begin{enumerate}
\item If $c_i\cap\gamma=\emptyset$, then $c_i$ and $\gamma$ co-bound an annulus since they are homotopic and disjoint in $\Sigma$. Since $x\notin N_0(c_i)$, we have $d(x,c_i)\geqslant\log\frac{2}{\ell(c_i)}$. By Lemma 2.4, we have
$$\sinh(r_0)>\frac14e^{d(x,c_i)}\ell(c_i)\geqslant\frac14\exp\bracket{\log\frac{2}{\ell(c_i)}}\ell(c_i)=\frac12$$
hence $r_0\geqslant\log\bracket{\frac{1+\sqrt5}2}$

\item 
If $c_i\cap\gamma\neq\emptyset$, let $x_0\in c_i\cap\gamma$. If $2r_0<0.97$, then $\ell(c_i)\leqslant\ell(\gamma)<0.97$, then $\log{\frac{2}{\ell(c_i)}}>\log2>0.69$. Hence 
$$d(c_i,x)\leqslant d(x_0,x)\leqslant\frac{\ell(\gamma)}2<0.485<0.69<\log{\frac{2}{\ell(c_i)}}$$
If follow that $x\in N_0(c_i)$ and we get a contradiction.
\end{enumerate}

\end{proof}

The proof of this lemma using the lemmas in \cite{Y2800}:

\begin{lemma}[\cite{Y2800}, Lemma 2.4]\label{lemma:injectivity_radius_annulus}

Let $A$ be an annulus in $\Sigma$ with boundary circles $\gamma_1$ and $\gamma_2$, where $\gamma_1$ is a geodesic and $\gamma_2$ is piecewise smooth. If there exists $x\in \gamma_2$ such that $d(\gamma_1, x)=d>0$, then
\begin{equation}
\sinh\bracket{\frac{\ell(\gamma_2)}2}>\frac14 e^d\ell(\gamma_1)
\end{equation}
If we further have $\ell(\gamma_1)<e^{-d}$, then
\begin{equation}
\ell(\gamma_2)>\frac{12}{25}e^{d}\ell(\gamma_1).
\end{equation}
\end{lemma}

\begin{lemma}[\cite{Y2800}, Lemma 2.5]\label{lemma:injectivity_radius_cusp}

Let $x\in{N_3(c_i)}$ for a cusp $c_i$ in $\mathcal{Y}$. If $d(x,\partial{N_3(c_i)})=d$, then the injective radius at $x$ is $\sinh^{-1}(e^{-d})$.
\end{lemma}

\section{Self-intersection estimate in thick part}\label{section3}

Let $\Gamma_2:=\Gamma\cap\caln_t$ be the part of $\Gamma$ in the thin part and $\Gamma_1:=\Gamma\setminus\Gamma_2$ in the thick part. Note that $\Gamma_2$ could be empty. Let $L_1:=\ell(\Gamma_1)$ and $L_2:=\ell(\Gamma_2)$, then $L=L_1+L_2$. The set $\cald$ of self-intersection points of $\Gamma$ consists of the self-intersection points of $\Gamma_1$ and $\Gamma_2$ since $\Gamma_1\cap\Gamma_2=\emptyset$. Let $\cald_1$ and $\cald_2$ be the sets of self-intersection points of $\Gamma_1$ and $\Gamma_2$ respectively. Then $\cald=\cald_1\cup\cald_2$ and
\begin{equation}\label{eqn:split_intersection_number}
\abs{\Gamma\cap\Gamma}=\abs{\Gamma_1\cap\Gamma_1}+\abs{\Gamma_2\cap\Gamma_2}
\end{equation}

In the case $\gamma_2\neq\emptyset$, since $\Gamma$ is closed, suppose $\Gamma_1$ is a collection of arcs $\delta_1,...,\delta_m$, $\Gamma_2$ is a collection of arcs $\gamma_1,...,\gamma_m$ where $m$ is the number of arcs. Since $d(x,\partial N_3(c_i))$ increases first and then decreases for each arc $\gamma'_0$ in $\Gamma\cap N_3(c_i)$, then $\gamma_0'\cap N_0(c_i)$ is either empty or an embedded arc, hence $\Gamma\cap\bracket{\bigcup_{c_i\in\calx\cup\caly}N_3(c_i)}$ is a collection of pairwisely disjoint arcs $\gamma_1',...,\gamma_n'$ where $m\leqslant n$. And for $1\leqslant k\leqslant m$, $\gamma_k\subseteq\gamma_k'$, for $m+1\leqslant k\leqslant n$, $\gamma_k'\subseteq\Gamma_1$. $\Gamma\setminus\bigcup_{j=1}^{m}\gamma_j'$ consists of $m$ arcs $\delta_1',...,\delta_m'$, for $1\leqslant j\leqslant m$, $\delta_j'\subseteq\delta_j$. Then if $\Gamma_2\neq\emptyset$, $\Gamma_1$ is also a colletion of $m$ arcs. We also have
$$L_1=\sum_{k=1}^{m}\ell(\delta_k)\qquad L_2=\sum_{k=1}^{m}\ell(\gamma_k)$$

We have two possiblities of each arc $\gamma_k$ of $\Gamma_2$:

\begin{enumerate}
 
\item \emph{General case}: $\Gamma$ intersects $N_0(c_i)$ for some $c_i\in\calx\cup\caly$ and $\Gamma$ intersects only one boundary of $\partial{N_0(c_i)}$. Assume In this case, as $x$ moves along $\gamma_k$, $d(x,\partial N_0(c_i))$ increases first and then decreases on arc $\gamma_k$, and $\gamma_k\cap c_i=\emptyset$.

\begin{figure}[htbp]
\centering
\tikzexternaldisable	
\begin{tikzpicture}[declare function={
	R = 42;
	f(\x) = R+.5-sqrt(R*R-\x*\x);
	ratio = 3;
}]
\def\x{2.5}
\draw [thick] (-4*\x,{f(4*\x)}) arc ({270-asin(4*\x/R)}:270:{R} and {R});
\draw [thick] (-4*\x,{-f(4*\x)}) arc ({90+asin(4*\x/R)}:90:{R} and {R});

\draw [thick, color=red] (-4*\x,0) ellipse ({f(4*\x)/ratio} and {f(4*\x)});

\foreach \a/\b in {-3.2*\x/red, 0/blue} {
	\draw [color=\b, thick] (\a,{-f(\a)}) arc (-90:90:{f(\a)/ratio} and {f(\a)});
	\draw [dashed, color=\b, thick] (\a,{f(\a)}) arc (90:270:{f(\a)/ratio} and {f(\a)});
}

\draw [thick, decoration={
	brace, mirror, raise=0.2cm
	}, decorate
] (-4*\x, {-f(4*\x)}) -- (0,{-f(3*\x)}) 
	node [pos=0.5,anchor=north,yshift=-0.2cm] {$N_3(c_i)$}; 

\draw [thick, decoration={
	brace, raise=0.1cm
	}, decorate
] (-3.2*\x,{f(3.2*\x)}) -- (0,{f(2*\x)}) 
	node [pos=0.5,anchor=south,yshift=0.1cm] {$N_0(c_i)$}; 

\draw[smooth, color=blue, thick] (-5*\x,0.4*\x)
  to [start angle=-20, next angle=0] ($(-4*\x,0)+(160:{f(4*\x)/ratio} and {f(4*\x)})$)
  to [next angle=40] ($(-4*\x,0)+(40:{f(4*\x)/ratio} and {f(4*\x)})$) coordinate(x);
\draw[smooth, color=blue, thick, dashed] (x) 
	to [start angle=40, next angle=-5] (-9,{f(9)}) coordinate (x);
\draw[smooth, color=blue, thick] (x) 
	to [start angle=-5, next angle=280] (-8.3,0) 
	to [next angle=5] (-7.8,{-f(7.8)}) coordinate (x);
\draw[smooth, color=blue, thick, dashed] (x) 
	to [start angle=5, next angle=80] (-7.35,0) 
	to [next angle=-5] (-6.7,{f(6.7)}) coordinate (x);
\draw[smooth, color=blue, thick] (x) 
	to [start angle=-5, next angle=280] (-6.1,0) 
	to [next angle=5] (-5.7,{-f(5.7)}) coordinate (x);
\draw[smooth, color=blue, thick, dashed] (x) 
	to [start angle=5, next angle=80] (-5.25,0) 
	to [next angle=-5] (-4.8,{f(4.8)}) coordinate (x);
\draw[smooth, color=blue, thick] (x) 
	to [start angle=-5, next angle=280] (-4.3,0) 
	to [next angle=5] (-4,{-f(4)}) coordinate (x);
\draw[smooth, color=blue, thick, dashed] (x) 
	to [start angle=5, next angle=80] (-3.65,0) 
	to [next angle=-5] (-3.3,{f(3.3)}) coordinate (x);
\draw[smooth, color=blue, thick] (x) 
	to [start angle=-5, next angle=270] (-2.8,0) 
	to [next angle=185] (-3.3,{-f(3.3)}) coordinate (x);
\draw[smooth, color=blue, thick, dashed] (x) 
	to [start angle=185, next angle=100] (-3.65,0) 
	to [next angle=175] (-4,{f(4)}) coordinate (x);
\draw[smooth, color=blue, thick] (x) 
	to [start angle=175, next angle=260] (-4.3,0) 
	to [next angle=185] (-4.8,{-f(4.8)}) coordinate (x);
\draw[smooth, color=blue, thick, dashed] (x) 
	to [start angle=185, next angle=100] (-5.25,0) 
	to [next angle=175] (-5.7,{f(5.7)}) coordinate (x);
\draw[smooth, color=blue, thick] (x) 
	to [start angle=175, next angle=260] (-6.1,0) 
	to [next angle=185] (-6.7,{-f(6.7)}) coordinate (x);
\draw[smooth, color=blue, thick, dashed] (x) 
	to [start angle=185, next angle=100] (-7.35,0) 
	to [next angle=175] (-7.8,{f(7.8)}) coordinate (x);
\draw[smooth, color=blue, thick] (x) 
	to [start angle=175, next angle=260] (-8.3,0) 
	to [next angle=185] (-9,{-f(9)}) coordinate (x);
\draw[smooth, color=blue, thick, dashed] (x) 
	to [start angle=185, next angle=140] ($(-4*\x,0)+(-40:{f(4*\x)/ratio} and {f(4*\x)})$) coordinate(x);
\draw[smooth, color=blue, thick] (x) 
	to [start angle=140, next angle=180] ($(-4*\x,0)+(200:{f(4*\x)/ratio} and {f(4*\x)})$)
	to [next angle=200] (-5*\x,-0.4*\x);

  			
\path ({f(0)/ratio},0) node[circle, inner sep=1pt, label={[label distance=.4em, anchor=center]0:$c_i$}]{};
  			
\end{tikzpicture}
\caption{\label{fig:general_case}
The general case}
\end{figure}
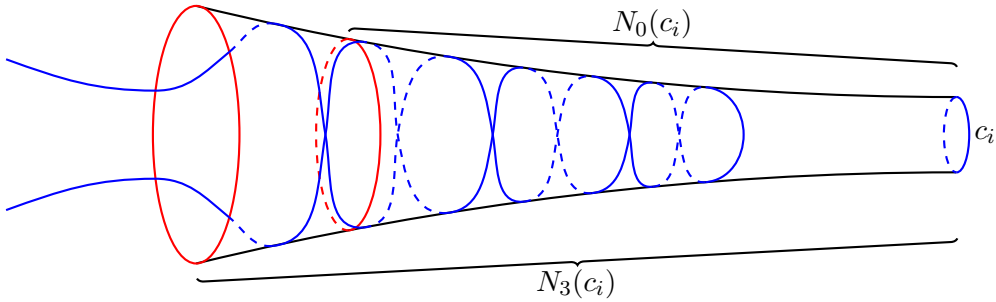

\item \emph{Special case}: $\Gamma$ intersects some $N_0(c_i)$ and $\Gamma$ intersects only both boundaries of $\partial{N_0(c_i)}$. In this case $\gamma_k$ is an embedded arc in $\Sigma$. Note that the special case does not happen for cusps.

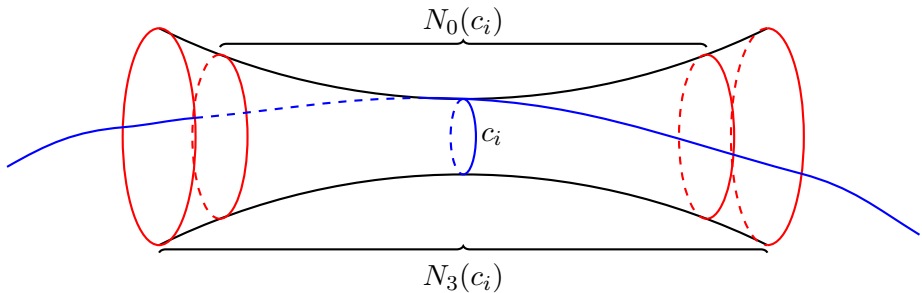
\begin{figure}[htbp]
\centering
\tikzexternaldisable
\begin{tikzpicture}[declare function={
	R = 9;
	f(\x) = R+.5-sqrt(R*R-\x*\x);
	ratio = 3;
}]
\def\x{1}
\draw [thick] (-4,{f(4)}) arc ({270-asin(4/R)}:{270+asin(4/R)}:{R} and {R});
\draw [thick] (-4,{-f(4)}) arc ({90+asin(4/R)}:{90-asin(4/R)}:{R} and {R});

\draw [thick, color=red] (-4,0) ellipse ({f(4)/ratio} and {f(4)});

\foreach \a/\b in {-3.2/red, 0/blue, 3.2/red, 4/red} {
	\draw [color=\b, thick] (\a,{-f(\a)}) arc (-90:90:{f(\a)/ratio} and {f(\a)});
	\draw [dashed, color=\b, thick] (\a,{f(\a)}) arc (90:270:{f(\a)/ratio} and {f(\a)});
}


\draw[smooth, color=blue, thick] (-6,-0.4)
  to [start angle=30, next angle=5] ($(-4*\x,0)+(175:{f(4*\x)/ratio} and {f(4*\x)})$)
	to [next angle=5] ($(-4*\x,0)+(10:{f(4*\x)/ratio} and {f(4*\x)})$) coordinate(x);
\draw[smooth, color=blue, thick, dashed] (x) 
	to [start angle=5, next angle=0] (-0.5,{f(-0.5)}) coordinate (x);
\draw[smooth, color=blue, thick] (x)
  to [start angle=0, next angle=-15] ($(4*\x,0)+(-20:{f(4*\x)/ratio} and {f(4*\x)})$)
  to [next angle=-32] (6, -1.3);

\draw [thick, decoration={
	brace, raise=0.1cm
	}, decorate
] (-3.2,{f(3.2)}) -- (3.2,{f(3.2)}) 
	node [pos=0.5,anchor=south,yshift=0.1cm] {$N_0(c_i)$}; 

\draw [thick, decoration={
	brace, raise=-0.1cm
	}, decorate
] (-4,-{f(4)}) -- (4,-{f(4)}) 
	node [pos=0.5,anchor=north,yshift=-0.1cm] {$N_3(c_i)$};

\path ({f(0)/ratio},0) node[circle, inner sep=1pt, label={[label distance=.4em, anchor=center]0:$c_i$}]{};
  			
\end{tikzpicture}

\caption{\label{fig:special_case}
The special case}
\end{figure}

\end{enumerate}

\begin{prop}
 $$m\leqslant\frac{L}{2\log2}$$
\end{prop}
\begin{proof}
For every $1\leqslant k\leqslant m$, there uniquely exists $c_i\in\calx\cup\caly$ such that $\gamma_k\subseteq N_0(c_i)$. Suppose the endpoints of $\gamma_k'$ are $a_k,b_k\in\partial N_3(c_i)$, and $x_k\in\gamma_k$, but $d(x_k,\partial N_3(c_i))>\log2$, so 
$$\ell(\gamma_k')\geqslant d(x_k,a_k)+d(x_k,b_k)\geqslant2\log2$$
Since $\gamma_1',...,\gamma_m'\subseteq\Gamma$ are disjoint arcs, so $m\leqslant\frac{L}{2\log2}$ holds.

\end{proof}

Similar as \cite[Theorem 3.3]{Y2800} we have the estimate of self-intersection number of $\Gamma_1$ as follows.

\begin{thm}
\begin{equation}\label{eqn:estimate_thick_part}
\abs{\Gamma_1\cap\Gamma_1}<\frac12\bracket{\frac{25}{12}L_1+m}^2
\end{equation}
\end{thm}

\begin{proof}

Recall that $\Gamma$ can be represented by local isometry $f:S^1\rightarrow\Sigma$. For $1\leqslant k\leqslant m$ divide $\delta_k$ into $M(k):=\gauss{\frac{\ell(\delta_k)}{0.48}}+1$ short closed segments with equal length 
$$\frac{\ell(\delta_k)}{M(k)}<0.48<\log\bracket{\frac{1+\sqrt5}2}$$
Suppose $M$ is the number of segments, then
$$M=\sum_{k=1}^{m}M(k)=\sum_{k=1}^{m}\gauss{\frac{\ell(\delta_k)}{0.48}}+1\leqslant\sum_{k=1}^{m}{\frac{\ell(\delta_k)}{0.48}}+1=\frac{25}{12}L_1+m$$
The set of these segments is $S$. Since $\calp_1$ is finite, $S$ can be chosen by small purturbing such that the endpoints of the segments do not contain points in $\calp_1$. 
$f(I_\alpha)\subset\Gamma_1$ is in the thick part and $\ell(I_\alpha)<0.48$, hence $f(I_\alpha)$ has no self-intersections by Lemma \ref{lemma:inj_radius_outside_nbhds}. 

We claim that any two distinct $I_\alpha, I_\beta\in S$ have at most one intersection. If there exist $s_1,s_2\in I_\alpha$ ($s_1\neq s_2$) and $t_1,t_2\in I_\beta$ such that $f(s_1)=f(t_1)$ and $f(s_2)=f(t_2)$. Let $\gamma_\alpha$ be the segment in $I_\alpha$ between $t_1$ and $t_2$, and $\gamma_\beta$ be the segment in $I_\beta$ between $s_1$ and $s_2$. Then $f(\gamma_\alpha)$ and $f(\gamma_\beta)$ are two distinct geodesics between $f(s_1)$ and $f(s_2)$. Since the injectivity radius at $f(s_1)=f(t_1)$ is at least $0.48$ and $\ell(f(\gamma_\alpha)),\ell(f(\gamma_\beta))<0.48$, the existence of $f(\gamma_\alpha)$ and $f(\gamma_\beta)$ contradicts the uniqueness of geodesics in $\bbh^2$.

Since any two distinct $I_\alpha, I_\beta\in S$ contribute at most 1 to $\abs{\Gamma_1\cap\Gamma_1}$, we have
$$\abs{\Gamma_1\cap\Gamma_1}\leqslant\frac12M(M-1)<\frac12M^2\leqslant \frac12\bracket{\frac{25}{12}L_1+m}^2$$

In addition, from the proof we know when $\Gamma_2=\emptyset$, we have
$$\abs{\Gamma\cap\Gamma}=\abs{\Gamma_1\cap\Gamma_1}\leqslant \frac12\bracket{\frac{25}{12}L+1}^2$$

\end{proof}

\section{Self-intersection estimate in thin part}\label{section4}

In this section, we give an estimate on the self-intersection number $|\Gamma_2\cap\Gamma_2|$. We have
$$\abs{\Gamma_2\cap\Gamma_2}=\sum_{p=1}^{m}\abs{\gamma_p\cap\gamma_p}+\sum_{1\leqslant p<q\leqslant m}\abs{\gamma_p\cap\gamma_q}$$

\subsection{Intersection number calculation}

We define the \emph{winding number} $w(\gamma_p)$ of the arc $\gamma_p$ for $1\leqslant p\leqslant m$, the winding number $w(\gamma_0)$ of any arc $\gamma_0\subseteq\gamma_p$, and $w(\gamma_p')$ can be similarly defined. Assume $\gamma_p\subseteq N_0(c_i)$. The definitions are similar as \cite{Y2805} and \cite{EP2802}.

\begin{enumerate}
\item
When $c_i\in\calx$ is a short geodesic, every point of $\gamma_p$ projects orthogonally to a well-defined point of $c_i$. The winding number of $\gamma_p$ is given by the quotient of the length of the projection of $\gamma_p$ divided by $\ell(c_i)$.

\item
When $c_i\in\caly$ is a cusp, every point of $\gamma_p$ projects orthogonally to a well-defined point of the length $h$ horocycle. The winding number of $\gamma_p$ is given by the quotient of the length of the projection of $\gamma_p$ divided by $h$.

\end{enumerate}

If $c_i\in\calx$ is a short geodesic, consider the Poincar\'{e} disk model of $\bbh^2$, the universal covering $p:\bbh^2\to\Sigma$ restricts to a universal covering $p:\Omega\to N_0(c_i)$ of $N_0(c_i)$. We may assume that $p^{-1}(c_i)$ is the horizontal line $\bbh^1\subset\bbh^2$. Let $\widetilde{\gamma_p}$, $\widetilde{\gamma_p'}$ be a lift of $\gamma_p$, $\gamma_p'$. Assume $P_1,P_2\in \widetilde{\gamma_p}$ are endpoints of $\widetilde{\gamma_p}$, $p(P_1),p(P_2)\in\partial N_0(c_i)$, and $P_1',P_2'\in \widetilde{\gamma_p'}$ are endpoints of $\widetilde{\gamma_p'}$, $p(P_1'),p(P_2')\in\partial N_3(c_i)$. $\widetilde{x'}$ is the midpoint of $P_1P_2$.

For $j=1,2$, define $Q_j, Q_j'\in\bbh^1$ are the unique points in $\bbh^1$ satisfies $d(P_j,Q_j)=d(P_j,\bbh^1)=\log\frac2{\ell(c_i)}$ and $d(P_j',Q_j')=d(P_j',\bbh^1)=\log\frac4{\ell(c_i)}$.

\begin{figure}[htbp]
\centering
\tikzexternaldisable
\begin{tikzpicture}
\def\x{4}
\coordinate (O) at (0,0);
\coordinate (P) at (-2.2*\x,0);
\coordinate (Q) at (2.2*\x,0);
\coordinate (K) at (0, 1.3*\x);

\begin{pgfinterruptboundingbox}
\path[name path=circle_O] (O) circle (\x);
\path[name path=PQ] (P) to (Q);
\coordinate[name intersections={of=circle_O and PQ, by={N,M}}];

\path[name path=circle_P] let
  \p1 = ($ (M) - (P) $),  \n1 = {veclen(\x1,\y1)},
  \p2 = ($ (N) - (P) $),  \n2 = {veclen(\x2,\y2)} in
  (P) circle ({sqrt(\n1)*sqrt(\n2)});
\coordinate[name intersections={of=circle_P and PQ, by={A}}];
\coordinate[name intersections={of=circle_P and circle_O, by={G,S}}];

\path[name path=circle_Q] let
  \p1 = ($ (M) - (Q) $),  \n1 = {veclen(\x1,\y1)},
  \p2 = ($ (N) - (Q) $),  \n2 = {veclen(\x2,\y2)} in
  (Q) circle ({sqrt(\n1)*sqrt(\n2)});
\coordinate[name intersections={of=circle_Q and PQ, by={B}}];
\coordinate[name intersections={of=circle_Q and circle_O, by={H,T}}];

\path[name path=circle_K] (K) circle ({sqrt(0.3*2.3)*\x});
\coordinate[name intersections={of=circle_P and circle_K, by={X,C}}];
\coordinate[name intersections={of=circle_Q and circle_K, by={X,D}}];
\coordinate[name intersections={of=circle_O and circle_K, by={E,F}}];

\end{pgfinterruptboundingbox}

\draw (O) circle (\x);
\draw[color=cyan, thin] (180:\x) to (0:\x);

\pgfmathanglebetweenpoints{\pgfpointanchor{P}{center}}{\pgfpointanchor{G}{center}}
\pgfmathsetmacro{\anglePG}{\pgfmathresult}
\draw[name path=arc_GS, color=cyan, thin]
  (-2.5,0) arc (0:-64.1:1.95);
\draw[name path=arc_GSS, color=cyan, thin]
  (-2.5,0) arc (0:64.1:1.95);
\draw[name path=arc_GS, color=cyan, thin]
  (2.5,0) arc (180:115.9:1.95);
\draw[name path=arc_GSS, color=cyan, thin]
  (2.5,0) arc (180:244.1:1.95);
 
\pgfmathanglebetweenpoints{\pgfpointanchor{K}{center}}{\pgfpointanchor{E}{center}}
\pgfmathsetmacro{\angleKE}{\pgfmathresult}
\draw[name path=arc_EF, color=cyan, thin]
  (0,0.5) arc (270:284.3:15.75);
\draw[name path=arc_EF, color=cyan, thin]
  (0,0.5) arc (270:255.7:15.75);
\draw[name path=arc_EF, color=red, thick]
  (0,0.5) arc (270:279.7:15.75);
\draw[name path=arc_EF, color=red,thick]
  (0,0.5) arc (270:260.3:15.75);

\draw[color=cyan] (O) to (0,0.5);

\node[below=0.3, anchor=center, color=blue] at (-0.8,1.2) {$\widetilde{\gamma}_p$};

\path (-2.5,0) node[circle, fill, inner sep=1pt, label={[shift={(305:0.45)}, anchor=center]{$Q_1$}}]{};
\path (2.5,0) node[circle, fill, inner sep=1pt, label={[shift={(235:0.45)}, anchor=center]{$Q_2$}}]{};
\path (-3,0) node[circle, fill, inner sep=1pt, label={[shift={(235:0.45)}, anchor=center]{$Q_1'$}}]{};
\path (3,0) node[circle, fill, inner sep=1pt, label={[shift={(305:0.45)}, anchor=center]{$Q_2'$}}]{};

\path (-2.2,1.3) node[circle, inner sep=1pt, label={[shift={(235:0.45)}, anchor=center]{$P_1$}}]{};
\path (2.2,1.3) node[circle, inner sep=1pt, label={[shift={(305:0.45)}, anchor=center]{$P_2$}}]{};
\path (-3.2,0.8) node[circle, fill, inner sep=1pt, label={[shift={(90:0.45)}, anchor=center]{$P_1'$}}]{};
\path (3.2,0.8) node[circle, fill, inner sep=1pt, label={[shift={(90:0.45)}, anchor=center]{$P_2'$}}]{};

\path (2.2,-1.3) node[circle, inner sep=1pt, label={[shift={(45:0.45)}, anchor=center]{$P_2''$}}]{};

\path (O) node[circle, fill, inner sep=1pt, label={[shift={(270:0.3)}, anchor=center]{$O$}}]{};
\path (0,0.5) node[circle, inner sep=1pt, label={[shift={(90:0.3)}, anchor=center]{$\widetilde{x'}$}}]{};

\draw [color=red, dashed, thin] (-2.6,0.7) to (2.6,-0.7);

\end{tikzpicture}
\caption{\label{fig:boundary_length_annulus}
A covering of $N_0(c_i)$ when $c_i\in\calx$}
\end{figure}

When $c_i$ is a cusp, consider the projection map $p$ from the upper half plane model of $\bbh^2$ to $\Sigma$, maps the ideal triangle $\infty A_0B_0$ to $N_0(c_i)$. Assume $A_0(-1,2)$, $B_0(1,2)$, $A_0'(-1,1)$, $B_0'(1,1)$ in Euclidean coordinate. $\widetilde{\gamma_p}$, $\widetilde{\gamma_p'}$ is a lift of the arc $\gamma_p$, $\gamma_p'$. Assume $\widetilde\gamma_p$ is the arc $x^2+y^2=R^2, y>2$ in Euclidean coordinate and $\widetilde\gamma_p$ is the arc $x^2+y^2=R^2, y>2$. $P_1(-\sqrt{R^2-4},2)$ $P_2(\sqrt{R^2-4},2)$ are endpoints of $\widetilde{\gamma_p'}$ and $P_1'(-\sqrt{R^2-1},1)$ $P_2'(\sqrt{R^2-1},1)$ are endpoints of $\widetilde{\gamma_p'}$. The hyperbolic length of the arc $P_1P_2$ is $\ell(P_1P_2)=\ell(\gamma_p)$. $\widetilde{x'}$ is the midpoint of $P_1P_2$.

\begin{lemma}\label{lemma_31}
For $1\leqslant p\leqslant m$, when $\gamma_p$ is of general case, we have
  $$\abs{\gamma_p\cap\gamma_p}<w(\gamma_p)$$
When $\gamma_p$ is of special case, 
$\abs{\gamma_p\cap\gamma_p}=0$.
\end{lemma}

\begin{proof}

  Notice that when $P$ goes from $p(P_1)$ to $p(P_2)$ along $\gamma_p$, then the function $d(P,\partial N_3(c_i))$ first increases and then decreases. Assume $P,P'\in\widetilde\gamma_p$, if $p(P)=p(P')$, then $\gamma_p$ must be of general case. $p(P)=p(P')$ if and only if the winding number of the arc $PP'\subseteq\gamma_p$ is a positive integer. Then the self intersection number of $\gamma_p$ equals to the number of positive integers less than $w(\gamma_p)$. Hence the lemma holds.
\end{proof}

\begin{figure}[htbp]
\begin{center}
\tikzexternaldisable
\begin{tikzpicture}
\def\x{1.5}

\draw[thick] (-1,0) to (-1,4*\x);
\draw[thick] (1,0) to (1,4*\x);

\draw[dashed] (4,0) arc (0:180:4);
\draw[dashed] (-4.5,1) to (4.5,1);
\draw[dashed] (-4.5,2) to (4.5,2);

\draw[dashed] (-3*\x,0) to (3*\x,0);

\draw[thick, color=red] (3.464,2) arc (30:150:4);

\path (0,0) node[circle, fill, inner sep=1pt, label={[shift={(270:0.45)}, anchor=center]{$O$}}]{};
\path (-1,1) node[circle, fill, inner sep=1pt, label={[shift={(135:0.45)}, anchor=center]{$A_0'$}}]{};
 \path (1,1) node[circle, fill, inner sep=1pt, label={[shift={(45:0.45)}, anchor=center]{$B_0'$}}]{};
\path (-1,2) node[circle, fill, inner sep=1pt, label={[shift={(135:0.45)}, anchor=center]{$A_0$}}]{};
 \path (1,2) node[circle, fill, inner sep=1pt, label={[shift={(45:0.45)}, anchor=center]{$B_0$}}]{};
\path (-3.464,2) node[circle, fill, inner sep=1pt, label={[shift={(135:0.45)}, anchor=center]{$P_1$}}]{};
 \path (3.464,2) node[circle, fill, inner sep=1pt, label={[shift={(45:0.45)}, anchor=center]{$P_2$}}]{};
\path (-3.873,1) node[circle, fill, inner sep=1pt, label={[shift={(135:0.45)}, anchor=center]{$P_1'$}}]{};
 \path (3.873,1) node[circle, fill, inner sep=1pt, label={[shift={(45:0.45)}, anchor=center]{$P_2'$}}]{};
\path (0,4) node[circle, fill, inner sep=1pt, label={[shift={(90:0.45)}, anchor=center]{$\widetilde{x'}$}}]{};

\draw (0,4*\x) node[above] {$\infty$};

	  			
\end{tikzpicture}
\end{center}
\tikzexternalenable
\caption{\label{fig:nbhd_cusps222}
A covering of $N_0(c_i)$ when $c_i\in\caly$}
\end{figure}

\begin{lemma}\label{lemma_32}

\begin{enumerate}
\item
   $w(\gamma_p)\leqslant 2\sinh\bracket{\frac{\ell(\gamma_p)}2}$. Similarly $w(\gamma_p')\leqslant \sinh\bracket{\frac{\ell(\gamma_p')}2}$.

\item
   For any $p$ such that $\gamma_p'\cap N_0(c_i)\neq\emptyset$ and $\gamma_p'$ is of general case, then $\ell(\gamma_p')\geqslant2\log(2+\sqrt3)$.

\item
For any $p$ such that $c_i\in\calx$, $\gamma_p\cap N_3(c_i)\neq\emptyset$, then $w(\gamma_p)\leqslant\frac{\ell(\gamma_p)}{\ell(c_i)}$.
\end{enumerate}
\end{lemma}

\begin{proof}

 \begin{enumerate}

\item
If $c_i\in\caly$ is a cusp, we have
$$\frac{\ell(\gamma_p)}2=\log\bracket{\frac{\sqrt{R^2-4}+R}{2}}=\log\bracket{\sqrt{(\frac{R}2)^2-1}+\frac{R}2}$$
Hence
$$w(\gamma_p)=\sqrt{R^2-4}=2\sinh\bracket{\frac{\ell(\gamma_p)}2}$$
On the other hand, when $\gamma_p=\gamma_p'\cap N_0(c_i)\neq\emptyset$, then $R>2$, hence
$$\ell(\gamma_p')=\ell(\widetilde{\gamma_p'})=\ell(P_1'P_2')=2\log(R+\sqrt{R^2-1})>2\log(2+\sqrt3)$$

\item
If $c_i\in\calx$ and $\gamma_p$ of general case,as in Figure \ref{fig:boundary_length_annulus}, $\widetilde{x'}$ is the midpoint of $P_1P_2$, since $p(P_1),p(P_2)\in\partial N_0(c_i)$, then $d(P,\bbh^1)=d(p(P),c_i)=\log\frac2{\ell(c_i)}$. From the definition of winding number, $d(Q_1,O)=\frac12d(Q_1,Q_2)=\frac{\ell(c_i)}2\cdot w(\gamma_p)$. The geodesic $P_1Q_1$ from $P_1$ to $Q_1$, horizontal line $Q_1O$, vertical line $O\widetilde{x'}$ and the left half of geodesic $\widetilde{\gamma_1}$ form a Lambert quadrilateril, the property of Lambert quadrilateril gives
\begin{align*}
\sinh\frac{\ell(\gamma_p)}2&=\sinh\bracket{\frac{\ell(c_i)}2\cdot w(\gamma_1)}\cosh\bracket{\log\frac2{\ell(c_i)}}\\
&\geqslant\frac{\ell(c_i)}2\cdot w(\gamma_1)\cdot\frac12\bracket{\frac2{\ell(c_i)}+\frac{\ell(c_i)}2}\geqslant\frac12w(\gamma_1)
\end{align*}
$w(\gamma_p)\leqslant\frac{\ell(\gamma_p)}{\ell(c_i)}$ follows from
\begin{align*}
\sinh\frac{\ell(\gamma_p)}2&=\sinh\bracket{\frac{\ell(c_i)}2\cdot w(\gamma_1)}\cosh\bracket{\log\frac2{\ell(c_i)}}\geqslant\sinh\bracket{\frac{\ell(c_i)}2\cdot w(\gamma_1)}
\end{align*}

On the other hand, if $\gamma_p=\gamma_p'\cap N_0(c_i)\neq\emptyset$, then $d(P_1',Q')=\log\frac4{\ell(c_i)}$ and $d(O,\widetilde{x'})<\log\frac2{\ell(c_i)}$. The property of Lambert quadrilateril gives
\begin{align*}
  \cosh{\frac{\ell(\gamma_p')}2}=\cosh{d(P_1',\widetilde{x'})}=\frac{\sinh{d(P_1',Q_1')}}{\sinh{d(O,\widetilde{x'})}}>\frac{\sinh{\log\frac4{\ell(c_i)}}}{\sinh{\log\frac2{\ell(c_i)}}}=\frac{16-\ell(c_i)^2}{8-2\ell(c_i)^2}>2
\end{align*}
Hence $\frac{\ell(\gamma_p')}2>\log(2+\sqrt3)$.

\item
If $c_i\in\calx$ and $\gamma_1$ of special case, as in Figure \ref{fig:boundary_length_annulus}, then $\widetilde{\gamma_1}$ is the dashed line $P_1OP_2''$, and $P_2''$ is the symmetry point of $P_2$, hence $P_1,P_2$ is on the same side of $\bbh^1$ and $P_2''$ is on the other side. We have
$$l_1=d(P_1,O)+d(O,P_2'')=d(P_1,O)+d(O,P_2)\geqslant d(P_1,P_2)$$
Hence from the proof of case 2 above
$$w(\gamma_1)\leqslant 2\sinh\bracket{\frac{d(P_1,P_2)}2}\leqslant 2\sinh\bracket{\frac{l_1}2}\qquad w(\gamma_1)\leqslant\frac{d(P_1,P_2)}{\ell(c_i)}<\frac{\ell(\gamma_1)}{\ell(c_i)}$$

\end{enumerate}

\end{proof}

If $\gamma_p,\gamma_q\subseteq N_0(c_i)$, we have conclusions about the intersection number $\abs{\gamma_p\cap\gamma_q}$:

\begin{lemma}
  Let $\gamma_1,\gamma_2$ be two distinct curves of general case in $\Gamma_2$ and $\gamma_3,\gamma_4$ be two distinct curves of special case, and $w(\gamma_1)\leqslant w(\gamma_2)$, $w(\gamma_3)\leqslant w(\gamma_4)$. Then
\begin{enumerate}
\item $\abs{\gamma_1\cap\gamma_3}\leqslant\lceil{w(\gamma_1)}\rceil$
\item $\abs{\gamma_1\cap\gamma_1}\leqslant\lceil{w(\gamma_1)}\rceil-1$
\item $\abs{\gamma_1\cap\gamma_2}\leqslant2\lceil{w(\gamma_1)}\rceil$
\item $\abs{\gamma_3\cap\gamma_4}\leqslant\lceil\frac{\ell(\gamma_3)+\ell(\gamma_4)}{\ell(c_i)}\rceil$
\end{enumerate}
\end{lemma}

\begin{proof}
 The first and the third inequality follows from \cite[Lemma 3.2]{EP2802} and the second is from Lemma \ref{lemma_31}. We only need to prove the fourth. Suppose $\gamma_3\subseteq N_0(c_i)$, $\gamma_4\subseteq N_0(c_{i'})$. If $i\neq i'$ then $\abs{\gamma_3\cap\gamma_4}=0$, next suppose $i=i'$.

Suppose $P_3,P_3'$ are endpoints of $\gamma_3$, and $P_4,P_4'$ are endpoints of $\gamma_4$, and $P_3,P_4$ are on the same boundary component of $\partial N_0(c_i)$ and $P_3',P_4'$ are on the another. Suppose the self-intersection points are $R_1,...,R_t\in\gamma_3$ from $P_3$ to $P_3'$, and define $\epsilon_k^3$, $\epsilon_k^4$ are the geodesic between $R_k,R_{k+1}$ along $\gamma_3$, $\gamma_4$. Since $\epsilon_k^3\cup\epsilon_k^4\subseteq N_0(c_i)$ is a closed curve freely homotopic to $c_i$, hence $\ell(\epsilon_k^3)+\ell(\epsilon_k^4)\geqslant\ell(c_i)$, hence
$$\abs{\gamma_3\cap\gamma_4}<1+\frac{\ell(\gamma_3)+\ell(\gamma_4)}{\ell(c_i)}$$
Since $\abs{\gamma_3\cap\gamma_4}$ is an integer, the fourth is proved.

\end{proof}

By abuse of notation assume there are $m_0$ arcs $\gamma_1,...,\gamma_{m_0}$ in $\Gamma_2$ with general case, suppose $w(\gamma_1)\leqslant...\leqslant w(\gamma_{m_0})$. And $m_2$ arcs $\gamma_{m_0+1},...,\gamma_{m}$ of special case. Set $m=m_0+m_2$. Define
$$\Gamma_2'=\bigcup_{j=1}^{m_0}\gamma_j\qquad \Gamma_2''=\bigcup_{j=m_0+1}^{m}\gamma_j\qquad L_2'=\sum_{j=1}^{m_0}\ell(\gamma_j)\qquad L_2''=\sum_{j=m_0+1}^{m}\ell(\gamma_j)$$
Clearly $\Gamma_2=\Gamma_2'\cup\Gamma_2''$, $L_2=L_2'+L_2''$. 
Similar as \cite[Lemma 2.5]{Y2805} we have
\begin{thm}
\begin{align}\label{ttt}
  \abs{\Gamma_2\cap\Gamma_2}\leqslant(2m-1)w(\gamma_1)+...+(2m_2+1)w(\gamma_{m_0})+{L_2''}e^{\frac{L_2''}4}+m^2-m
\end{align}
\end{thm}

\begin{proof}

Suppose $C:=\{c_1,...,c_l\}\subseteq\calx$ is the set $\{c\in\calx:c\cap\Gamma_2''\neq\emptyset\}$. For $1\leqslant j\leqslant l$, define $C_j$ is the collection of arcs in $\Gamma_2''$ intersecting $c_j$, assume $m_{2,j}:=\#(C_j)$ and $\Gamma''_{2,j}:=\bigcup_{\gamma\in C_j}\gamma$, $L''_{2,j}:=\sum_{\gamma\in C_j}\ell(\gamma)$. Clearly $m_2=m_{2,1}+...+m_{2,l}$ and $L_2''=L_{2,1}''+...+L_{2,l}''$. Notice that for each $\gamma\in\Gamma_{2,j}''$, since $\gamma$ is of special case we have $\ell(\gamma)\geqslant2\log\frac2{\ell(c_j)}$, hence we have $\ell(c_j)\geqslant2\exp(-{\frac{L_{2,j}''}{2m_{2,j}}})$. By definition of $\calx$ we have $2\log\frac2{\ell(c_j)}>2\log2$, hence $2\log2\cdot{m_{2,j}}\leqslant L_{2,j}''$. As a result,

\begin{align*}
  \abs{\Gamma_2\cap\Gamma_2}&=\sum_{j=1}^{m}\abs{\gamma_j\cap\gamma_j}+\sum_{1\leqslant j_1<j_2\leqslant m}\abs{\gamma_{j_1}\cap\gamma_{j_2}}\\
&\leqslant\sum_{j=1}^{m_0}(\lceil{w(\gamma_j)}\rceil-1)+\sum_{1\leqslant j_1<j_2\leqslant m, j_1\leqslant m_0}2\lceil{w(\gamma_{j_1})}\rceil
+\sum_{j=1}^{l}\sum_{\{\gamma',\gamma''\}\subseteq C_j}1+\frac{\ell(\gamma')+\ell(\gamma'')}{\ell(c_j)}\\
&\leqslant\sum_{j=1}^{m_0}(2m+1-2j)(w(\gamma_j)+1)+\sum_{j=1}^{l}(m_{2,j}-1)\sum_{\gamma\in C_j}\frac{\ell(\gamma)}{\ell(c_j)}+\sum_{j=1}^{l}\frac{m_{2,j}(m_{2,j}-1)}2-m_0\\
&\leqslant\sum_{j=1}^{m_0}(2m+1-2j)w(\gamma_j)+m^2-m_2^2+\sum_{j=1}^{l}(m_{2,j}-1)\frac{L_{2,j}''}{\ell(c_j)}+m_2^2-m_2-m_0\\
&\leqslant\sum_{j=1}^{m_0}(2m+1-2j)w(\gamma_j)+\sum_{1\leqslant j\leqslant l, m_{2,j}\geqslant2}\frac{m_{2,j}-1}2L_{2,j}''e^{\frac{L_{2,j}''}{2m_{2,j}}}+m^2-m\\
&\leqslant\sum_{j=1}^{m_0}(2m+1-2j)w(\gamma_j)+\sum_{1\leqslant j\leqslant l, m_{2,j}\geqslant2}L_{2,j}''e^{\frac{L_{2,j}''}{4}}+m^2-m\\
&\leqslant\sum_{j=1}^{m_0}(2m+1-2j)w(\gamma_j)+{L_2''}e^{\frac{L_2''}4}+m^2-m
\end{align*}
The last but one inequality using the fact that when $x\geqslant2$, $a\geqslant 2x\log2$, then $(x-1)e^{\frac{a}{2x}}\leqslant2e^{\frac{a}4}$.
\end{proof}

\subsection{Upper bound estimate for $|\Gamma_2\cap\Gamma_2|$}

In order to estimate the upper bound of $\abs{\Gamma_2\cap\Gamma_2}$ we need to estimate (10). Define function $D:\mathbb{R}^m\rightarrow\mathbb{R}$:
$$D(x_1,...,x_{m_0}):=(2m-1)\sinh{x_1}+...+(2m_2+3)\sinh{x_{m_0-1}}+(2m_2+1)\sinh{x_{m_0}}$$

\begin{lemma}
Define $A\subseteq\mathbb{R}^m$:
$$A:=\set{(x_1,...,x_{m_0}): 0\leqslant x_1\leqslant...\leqslant x_m\qquad x_1+...+x_{m_0}=L_2'}$$
If function $D$ attains its maximum on $x'=(x_1',...,x_{m_0}')\in\mathbb{R}^m$, then there exists integer $1\leqslant m_1\leqslant m_0$, such that 
$$x'=\bracket{0,0,...,0,\frac{L_2'}{m_1},...,\frac{L_2'}{m_1}}$$
Here the number of 0 is $m_0-m_1$.
\end{lemma}

\begin{proof}
 Since $A$ is compact, we know the maximum point $x'$ of function $D$ in $A$ exists. We prove the lemma by contradiction. Otherwise there exists $2\leqslant s\leqslant m_0$ such that $0<x_{s-1}'<x_s'$. Let $u,v\geqslant1$ be maximal integers satisfying $0<x_{s-u}'=...=x_{s-1}'<x_s'=...=x_{s+v-1}'$. Choose $\epsilon>0$ small that $v\epsilon<x_{s-u}'$, $2uv\epsilon<x_s'-x_{s-1}'$, and if $s-u>1$ then $v\epsilon<x_{s-u}'-x_{s-u-1}'$, if $s+v-1<m_0$ then $u\epsilon<x_{s+v}'-x_{s+v-1}'$. For $0\leqslant\delta\leqslant\epsilon$ define function
$$H(\delta):=\sum_{k=s-u}^{s-1}(2m+1-2k)\sinh(x_k'-v\delta)+\sum_{k=s}^{s+v}(2m+1-2k)\sinh(x_k'+u\delta)$$
Since $x'$ is the maximal point of $D$, we have $H'(0)=0$ and $H''(0)\leqslant0$, but the function $f(x)=\sinh{x}$ is a convex function, we have $H''(0)>0$, a contradiction. 

\end{proof}

As a simple corollary we have

\begin{cor}\label{m_bigger2_1}
 If $\Gamma\cap\caln_t\neq\emptyset$ and $m_2=0$, i.e. $m\geqslant1$, then there exists $1\leqslant m_1\leqslant m_0$ that
\begin{align}\label{fff12}
\abs{\Gamma\cap\Gamma}<\frac12\bracket{\frac{25}{12}L_1+m}^2+m_1^2e^{\frac{L_2}{2m_1}}+m^2-m
\end{align}
If $m_2\geqslant1$ then
\begin{align}\label{fff12_1}
\abs{\Gamma\cap\Gamma}<\frac12\bracket{\frac{25}{12}L_1+m}^2+(m_1^2+2m_1m_2)e^{\frac{L_2'}{2m_1}}+{L_2''}e^{\frac{L_2''}4}+m^2-m
\end{align}
\end{cor}

\begin{proof}
 When $m_2=0$, combining with Theorem 2.7 and Theorem 3.4 there exists such $m_1$ that
\begin{align*}
\abs{\Gamma\cap\Gamma}&\leqslant\frac12\bracket{\frac{25}{12}L_1+m}^2+2D\bracket{\frac{\ell(\gamma_1)}2,...,\frac{\ell(\gamma_m)}2}+m^2-m\\
&\leqslant\frac12\bracket{\frac{25}{12}L_1+m}^2+2m_1^2\sinh\frac{L_2}{2m_1}+m^2-m<\frac12\bracket{\frac{25}{12}L_1+m}^2+m_1^2e^{\frac{L_2}{2m_1}}+m^2-m
\end{align*}
When $m_2\geqslant1$ the proof is same.
\end{proof}

\section{Total intersection number estimate}\label{section5}

In this section, we complete the proof of Theorem \ref{thm:main} by examining two distinct cases. For $m \geq 2$, the geodesic cannot penetrate deeply into the thin part, which ensures a controlled self-intersection pattern. When $m = 1$, by minimizing the length in the thick region, we allow the geodesic to extend further into the thin part, thereby increasing its self-intersections. In this configuration, the geodesic necessarily adopts a corescrew structure, whose well-defined geometry enables precise determination of the intersection count through systematic examination. This comprehensive case analysis establishes the desired result.

\subsection{Case $m\geqslant2$}

In this subsection we consider the case when $m\geqslant2$:

\begin{thm}\label{m_bigger2} 
 When $k>1750$ and $m\geqslant2$, if $\abs{\Gamma\cap\Gamma}=k$, then $L\geqslant2\cosh^{-1}(2k+1)$. 
\end{thm}

The proof is organized into three main components, each formulated as a separate theorem: Theorems~\ref{m_bigger2_01}, \ref{m_bigger2_02}, and~\ref{m_bigger2_03} below.

\begin{thm}\label{m_bigger2_01} 
 When $k>1750$ and $m\geqslant2$, if $m_2=0$ and $\abs{\Gamma\cap\Gamma}=k$, then $L\geqslant2\cosh^{-1}(2k+1)$. 
\end{thm}

\begin{proof}
 Consider the function
$$I(m,m_1,L_2)=\frac12\bracket{\frac{25}{12}(L-L_2)+m}^2+m_1^2e^{\frac{L_2}{2m_1}}+m^2-m$$
where $1\leqslant m_1\leqslant m<\frac{L}{2\log2}$, $0\leqslant L_2\leqslant L-2m\log2$, $L_1+L_2=L$.
We have
\begin{align*}
\frac{\partial^2I}{\partial m_1^2}&=\frac{\partial}{\partial m_1}\bracket{\bracket{2m_1-\frac{L_2}2}e^{\frac{L_2}{2m_1}}}=\bracket{2-\frac{L_2}{2m_1^2}\bracket{2m_1-\frac{L_2}2}}e^{\frac{L_2}{2m_1}}\\
&=\bracket{1+\bracket{\frac{L_2}{2m_1}-1}^2}e^{\frac{L_2}{2m_1}}\geqslant0
\end{align*}
So $I(m,m_1,L_2)$ is a convex function of variable $m_1$. Hence
$$I(m,m_1,L_2)\leqslant\max\{I(m,m,L_2),I(m,1,L_2)\}$$
\begin{enumerate}
\item
When $m_1=1$, we have
\begin{align*}
 \frac{\partial^2I(m,1,L_2)}{\partial L_2^2}=\frac12\frac{\partial^2}{\partial L_2^2}\bracket{\frac{25}{12}(L-L_2)+m}^2+\frac{\partial^2}{\partial L_2^2}e^{\frac{L_2}{2}}=\frac{625}{144}+\frac14e^{\frac{L_2}2}>0
\end{align*}
Using the fact $L_2\leqslant L-2m\log2$, if $m\geqslant3$ we have
\begin{align*}
 I(m,1,L_2)\leqslant \max&\{I(m,1,0),I(m,1,L-2m\log2)\}\\
\leqslant\max&\left\{\frac12\bracket{\frac{25}{12}+\frac1{2\log2}}^2L^2+m^2-m+1,\frac12\bracket{\frac{6+25\log2}6m}^2+\frac1{2^m}e^{\frac{L}2}+m^2-m\right\}\\
\leqslant \max&\left\{\frac12\bracket{\frac{25}{12}+\frac1{2\log2}}^2L^2+(\frac{L}{2\log2})^2,\frac12\bracket{\frac{6+25\log2}2}^2+\frac1{8}e^{\frac{L}2}+6\right\}\\
\leqslant\max&\left\{75+\frac18e^{\frac{L}2}, 4.53L^2\right\}
\end{align*}
The third inequality using the fact that both two functions in "max" are convex on $m\in[3,\frac{L}{2\log2}]$. If $L<2\cosh^{-1}(2k+1)$ then $e^{\frac{L}2}<4k+2$, but when $L\geqslant17.2$ we have $e^{\frac{L}2}>18.12L^2+2$, $302+\frac12e^{\frac{L}2}<e^{\frac{L}2}$, contradiction with $4\abs{\Gamma\cap\Gamma}+2>e^{\frac{L}2}$. Then $L<17.2$, hence 
$$\max\left\{4.46L^2, 75+\frac18e^{\frac{L}2}\right\}<1600$$
we get a contradiction.

If $m=2$ and $L_2\leqslant L-2m\log2-0.48$, then 
\begin{align*}
 I(m,1,L_2)\leqslant \max&\{I(m,1,0),I(m,1,L-4\log2-0.48)\}\\
\leqslant\max&\left\{\frac12\bracket{\frac{25}{12}+\frac1{2\log2}}^2L^2+m^2-m+1,\frac1{2^me^{0.24}}e^{\frac{L}2}+41\right\}\\
\leqslant \max&\left\{4.46L^2+3,\frac{0.79}{4}e^{\frac{L}2}+41\right\}
\end{align*}
Since $4\abs{\Gamma\cap\Gamma}+2\leqslant\max\{17.84L^2+14, 166+{0.79}e^{\frac{L}2}\}<e^{\frac{L}2}$ for $L>17.7$ we get a contradiction. When $L\leqslant17.7$, $\max\left\{4.46L^2+3, 41+\frac{0.79}{4}e^{\frac{L}2}\right\}<1750$, contradiction.

If $m=2$ and $L_2>L-4\log2-0.48$, assume $L_1':=\ell(\gamma_1'), L_2':=\ell(\gamma_2'), L':=L_1'+L_2', L_1'\leqslant L_2'$. Since there exists $c_{i'},c_i\in\calx\cup\caly$, such that $\gamma_1'\cap N_0(c_{i'})\neq\emptyset$, $\gamma_2'\cap N_0(c_{i})\neq\emptyset$. if $L_1',L_2'\geqslant2\log(2+\sqrt3)$, 
using Lemma \ref{lemma_32} and \cite[Lemma 2.5]{Y2805} we have
\begin{align*}
 \abs{\Gamma\cap\Gamma}&\leqslant\frac12\bracket{\frac{4\log2+0.48}{0.48}+2}^2+3\sinh{\frac{L_1'}2}+\sinh\frac{L_2'}2+2\\
&\leqslant\max\left\{41+4\sinh{\frac{L}4}, 40.6+3\sinh(\log(2+\sqrt3))+\sinh\bracket{\frac{L}2-\log(2+\sqrt3)}\right\}\\
&\leqslant\max\left\{41+2e^{\frac{L}4}, 46+\frac17e^{\frac{L}2}\right\}
\end{align*}
Since $166+8e^{\frac{L}4}<e^{\frac{L}2}$ and $186+\frac47e^{\frac{L}2}<e^{\frac{L}2}$ holds for $L>17$, contradicting $4\abs{\Gamma\cap\Gamma}+2>e^{\frac{L}2}$. When $L\leqslant17$, $\max\left\{41+2e^{\frac{L}4}, 46+\frac17e^{\frac{L}2}\right\}<1600$, contradiction.

If $L_1'<2\log(2+\sqrt3)$, using Lemma \ref{lemma_32}, $\gamma_1'$ cannot be general case, it must be special case. Hence $2\log\frac4{\ell(c_{i'})}\leqslant2\log(2+\sqrt3)<4\log2$, we have $\ell(c_{i'})>1$, contradiction to the definition of $\calx$ and $\caly$.

\item

When $m_1=m$, we have
\begin{align*}
\frac{\partial^2I(m,m,L_2)}{\partial m^2}&=2+\frac{\partial^2}{\partial m^2}m^2e^{\frac{L_2}{2m}}=2+\bracket{1+\bracket{\frac{L_2}{2m}-1}^2}e^{\frac{L_2}{2m}}>0\\
\frac{\partial^2I(m,m,L_2)}{\partial L_2^2}&=\frac12\frac{\partial^2}{\partial L_2^2}\bracket{\frac{25}{12}(L-L_2)+m}^2+\frac{\partial^2}{\partial L_2^2}m^2e^{\frac{L_2}{2m}}=\frac{625}{144}+\frac14e^{\frac{L_2}{2m}}>0
\end{align*}
hence
\begin{align*}
 I(m,m,L_2)\leqslant&\max\left\{I(2,2,L_2),I(\frac{L}{2\log2},\frac{L}{2\log2},L_2)\right\}\\
\leqslant&\max\{I(2,2,0),I(2,2,L-4\log2),I(\frac{L}{2\log2},\frac{L}{2\log2},0),I(\frac{L}{2\log2},\\
&\frac{L}{2\log2},L-4\log2)\}\\
\leqslant&\max\{\frac12\bracket{\frac{25}{12}+\frac1{2\log2}}^2L^2+(\frac{L}{2\log2})^2+4e^{\frac{L}4}, \frac12\bracket{\frac{25}{12}+\frac1{2\log2}}^2L^2+3(\frac{L}{2\log2})^2\}\\
\leqslant&\max\{4.454L^2+4e^{\frac{L}4}, 5.5L^2\}
\end{align*}
When $L>17.7$, we have $4(4.454L^2+4e^{\frac{L}4})+2<e^{\frac{L}2}$, $22L^2+2<e^{\frac{L}2}$, a contradiction. When $L\leqslant17.7$, we have $4.454L^2+4e^{\frac{L}4}<1750$ and $5.5L^2<1750$, contradiction.

\end{enumerate}

\end{proof}
Note that when $m=0$, i.e. $\Gamma\cap\caln_t=\emptyset$, then Lemma \ref{m_bigger2_1} holds, hence we have
$$k=\abs{\Gamma\cap\Gamma}\leqslant\frac12\bracket{\frac{25}{12}L+1}^2$$
Hence when $k>1700$, then $L>17$, hence $4k+2\leqslant2+2\bracket{\frac{25}{12}L+1}^2<e^{\frac{L}2}$, contradiction.

\begin{thm}\label{m_bigger2_02} 
 When $k>1750$ and $m\geqslant2$, if $m_0=0$ and $\abs{\Gamma\cap\Gamma}=k$, then $L\geqslant2\cosh^{-1}(2k+1)$. 
\end{thm}

\begin{proof}
When $m_0=0$ then $m=m_2$ and $m_2\geqslant2$, we have
\begin{align*}
k=\abs{\Gamma\cap\Gamma}&<\frac12\bracket{\frac{25}{12}L_1+m}^2+{L_2''}e^{\frac{L_2''}4}+m^2-m\\
&\leqslant\frac12\bracket{\frac{25}{12}L+m}^2+({L-2m\log2})e^{\frac{L-2m\log2}4}+m^2-m\\
&\leqslant\max\left\{\frac12\bracket{\frac{25}{12}L+2}^2+(L-4\log2)e^{\frac{L-4\log2}4}+2, \frac12\bracket{\frac{25}{12}+\frac1{2\log2}}^2L^2+\bracket{\frac{L}{2\log2}}^2\right\}\\
&\leqslant\max\left\{\frac12\bracket{\frac{25}{12}L+2}^2+\frac{L}2e^{\frac{L}4}+2,4.46L^2\right\}
\end{align*}
The third inequality using the fact that the function before the inequality sign is convex on $m\in(2,\frac{L}{2\log2})$. 
When $L>17.7$, we have $4k+2<e^{\frac{L}2}$. When $L\leqslant17.7$ we have $k<1750$, contradiction.

\end{proof}

\begin{thm}\label{m_bigger2_03} 
 When $k>1750$ and $m\geqslant2$, if $m_0,m_2\geqslant1$ and $\abs{\Gamma\cap\Gamma}=k$, then $L\geqslant2\cosh^{-1}(2k+1)$. 
\end{thm}

\begin{proof}
Recall that $\abs{\Gamma\cap\Gamma}=I(L_2',L_2'',m,m_1,m_2)$, where $L_1=L-L_2-L_2''$ and
\begin{align*}
I(L_2',L_2'',m,m_1,m_2)=\frac12\bracket{\frac{25}{12}L_1+m}^2+(m_1^2+2m_1m_2)e^{\frac{L_2'}{2m_1}}+{L_2''}e^{\frac{L_2''}4}+m^2-m
\end{align*}
For $a>0$ functions $f(x)=xe^{\frac{a}{x}}$ and $f(x)=x^2e^{\frac{a}{x}}$ are convex, hence $I(L_2',L_2'',m,m_1,m_2)$ is a convex function for $m_1\in(1,m_0)$. Hence one of the following holds:
\begin{align}\label{0001}
I\leqslant\frac12\bracket{\frac{25}{12}L_1+m}^2+(1+2m_2)e^{\frac{L_2'}{2}}+{L_2''}e^{\frac{L_2''}4}+m^2-m
\end{align}
\begin{align}\label{0002}
I\leqslant\frac12\bracket{\frac{25}{12}L_1+m}^2+(m_0^2+2m_0m_2)e^{\frac{L_2'}{2m_0}}+{L_2''}e^{\frac{L_2''}4}+m^2-m
\end{align}

\begin{enumerate}
\item
If (\ref{0001}) holds, since for $m_0+1\leqslant j\leqslant m$, $\ell(\gamma_j)\geqslant2\log2$, we have $L_2'\leqslant L_2-2m_2\log2\leqslant L-2m\log2-2m_2\log2$.  then when $m\geqslant7$ we have $m+m_2\geqslant8$, hence
\begin{align*}
 I&\leqslant\frac12\bracket{\frac{25}{12}L+m}^2+\frac{1+2m_2}{2^{m+m_2}}e^{\frac{L}{2}}+{L_2''}e^{\frac{L_2''}4}+m^2-m\\
&\leqslant\frac12\bracket{\frac{25}{12}L+m}^2+\frac{3}{256}e^{\frac{L}{2}}+(L-2m\log2)e^{\frac{L-2m\log2}4}+m^2-m\\
&\leqslant\max\left\{\frac12\bracket{\frac{25}{12}L+7}^2+\frac{3}{256}e^{\frac{L}{2}}+Le^{\frac{L-14\log2}4}+42,\frac12\bracket{\frac{25}{12}L+\frac{L}{2\log2}}^2+\frac{3}{256}e^{\frac{L}{2}}+\bracket{\frac{L}{2\log2}}^2\right\}\\
&\leqslant\max\left\{\frac12\bracket{\frac{25}{12}L+7}^2+\frac{3}{256}e^{\frac{L}{2}}+\frac{L}{8\sqrt2}e^{\frac{L}4}+42,\frac{3}{256}e^{\frac{L}{2}}+4.46L^2\right\}
\end{align*}
The second inequality uses the fact that when $m>m_2\geqslant1$ and $m\geqslant7$, then $\frac{1+2m_2}{2^{m+m_2}}\leqslant\frac3{256}$. The third inequality uses the convexity on $m\in(7,\frac{L}{2\log2})$. When $L>17.7$ we have $4I+2\leqslant e^{\frac{L}2}$, and when $L\leqslant17.7$ we have $4.46L^2+\frac{3}{256}e^{\frac{L}{2}}<1700$ and $\frac12\bracket{\frac{25}{12}L+7}^2+\frac{3}{256}e^{\frac{L}{2}}+\frac{L}{8\sqrt2}e^{\frac{L}4}+42<1700$, contradiction.

When $4\leqslant m\leqslant6$, using the convexity of $L_2\in(0,L-2m\log2)$, and $L_2\leqslant L-2m\log2\leqslant L-8\log2$, we have
\begin{align*}
 I&\leqslant\frac12\bracket{\frac{25}{12}(L-L_2)+m}^2+(1+2m_2)e^{\frac{L_2-2m_2\log2}2}+{L_2}e^{\frac{L_2}4}+m^2-m\\
&\leqslant\max\left\{m^2-m+2m_2+1+\frac12\bracket{\frac{25}{12}L+6}^2, \frac12\bracket{\frac{25\log2}{6}m+m}^2+\frac{3}{32}e^{\frac{L}{2}}+\frac{L}{4}e^{\frac{L}4}+m^2-m\right\}\\
&\leqslant\max\left\{44+\frac12\bracket{\frac{25}{12}L+6}^2, 303+\frac{3}{32}e^{\frac{L}{2}}+\frac{L}{4}e^{\frac{L}4}\right\}
\end{align*}
since $\frac{1+2m_2}{2^{m+m_2}}\leqslant\frac3{32}$ when $m\geqslant4, m_2\geqslant1$. Hence when $L>17.7$ then $4I+2\leqslant e^{\frac{L}2}$, when $L\leqslant17.7$ then $I\leqslant1750$, contradiction.

When $m=3$, $m_1\geqslant1$ implies $\ell(\gamma_1')\geqslant2\log(2+\sqrt3)>2.62$ using Lemma \ref{lemma_32}, hence $L_2''\leqslant L-4\log2-2\log(2+\sqrt3)\in(L-5.41,L-5.4)$. Using the convexity of $L_2\in(0,L-6\log2)$, we have
\begin{align*}
 I&\leqslant\frac12\bracket{\frac{25}{12}(L-L_2)+3}^2+(1+2m_2)e^{\frac{L_2-2m_2\log2}{2}}+{L_2}e^{\frac{L-4\log2-2\log(2+\sqrt3)}4}+6\\
&\leqslant\max\left\{13+\frac12\bracket{\frac{25}{12}L+6}^2,\frac12\bracket{\frac{25\times6\log2}{12}+3}^2+\frac{3}{16}e^{\frac{L}{2}}+\frac{L-6\log2}{2\sqrt{2+\sqrt3}}e^{\frac{L}4}+6\right\}\\
&\leqslant\max\left\{13+\frac12\bracket{\frac{25}{12}L+6}^2, 110+\frac{3}{16}e^{\frac{L}{2}}+\frac{L-6\log2}{3.86}e^{\frac{L}4}\right\}
\end{align*}
since $\frac{1+2m_2}{2^{m+m_2}}\leqslant\frac3{16}$ when $m\geqslant3, m_2\geqslant1$. Hence when $L>17.7$ then $4I+2\leqslant e^{\frac{L}2}$, when $L\leqslant17.7$ then $I\leqslant1750$, contradiction.

When $m_1=m_2=1$, then suppose $\gamma_1$ is of general case and $\gamma_2$ is of special case, and $\gamma_1\cap N_0(c_i)\neq\emptyset$, $\gamma_2\cap N_0(c_{i'})\neq\emptyset$. If $i\neq i'$ then $\abs{\gamma_2\cap\gamma_2}=\abs{\gamma_1\cap\gamma_2}=0$, same as Theorem \ref{m_bigger2_01} to get the proof. Otherwise $i=i'$. 

If $\ell(c_i)\geqslant 0.25$ then $w(\gamma_1)\leqslant\frac{L_2}{0.25}<4L$, and $\abs{\Gamma_2\cap\Gamma_2}\leqslant2w(\gamma_1)+4$, hence 
$$k=\abs{\Gamma\cap\Gamma}\leqslant\frac12\bracket{\frac{25}{12}L+2}^2+8L+4$$ 
When $L>17.7$ we have $4(\frac12\bracket{\frac{25}{12}L+2}^2+8L+4)+2<e^{\frac{L}2}$, when $L\leqslant17.7$ we have $k<1750$, ontradiction.

If $\ell(c_i)<0.25$, we have $\ell(\gamma_2)\geqslant6\log2$, hence $\ell(\gamma_1)\leqslant L-10\log2$. Hence using Lemma \ref{lemma_32} we have $w(\gamma_1)\leqslant\frac1{32}e^{\frac{L}2}$, then
$$k=\abs{\Gamma\cap\Gamma}\leqslant\frac12\bracket{\frac{25}{12}L+2}^2+\frac1{16}e^{\frac{L}2}+4$$
When $L>17.7$ we have $4k+2<e^{\frac{L}2}$, when $L\leqslant17.7$ we have $k<1750$, ontradiction.

\item
If (\ref{0002}) holds and $m_0\geqslant2$, then $\frac{L_2'}{2m_0}\leqslant\frac{L_2'}4$, hence when $L\geqslant16.7$ we have
\begin{align*}
I&\leqslant\frac12\bracket{\frac{25}{12}L_1+m}^2+m^2e^{\frac{L_2'}{2m_0}}+{L_2''}e^{\frac{L_2''}4}+m^2-m\\
&\leqslant\frac12\bracket{\frac{25}{12}L+m}^2+\frac{m^2+L}{2^{\frac{m}2}}e^{\frac{L}4}+m^2-m\\
&\leqslant\max\left\{\frac12\bracket{\frac{25}{12}L+2}^2+(2+\frac{L}2)e^{\frac{L}4}+2, \frac12\bracket{\frac{25}{12}L+\frac{L}{2\log2}}^2+2\bracket{\frac{L}{2\log2}}^2+L\right\}
\end{align*}
The third inequality using the convexity on $m\in(2,\frac{L}{2\log2})$, here when $L\geqslant16.7$ the function $f(x)=\frac{x^2+L}{2^{\frac{x}2}}$ is convex for $x\in(2,+\infty)$. When $L\geqslant17.7$ then $4I+2\leqslant e^{\frac{L}2}$, and when $16.7\leqslant L<17.7$, $I<1750$, contradiction.

When $L<16.7$, since for $x>1$, $\frac{x^2}{2^{\frac{x}2}}<4.51$, then using the convexity on $m\in(2,\frac{L}{2\log2})$ we have
\begin{align*}
I&\leqslant\frac12\bracket{\frac{25}{12}L_1+m}^2+(4.51+\frac{L}{2^{\frac{m}2}})e^{\frac{L}4}+m^2-m\\
&\leqslant\max\left\{\frac12\bracket{\frac{25}{12}L+2}^2+(4.51+\frac{L}2)e^{\frac{L}4}+2, \frac12\bracket{\frac{25}{12}L+\frac{L}{2\log2}}^2+\bracket{\frac{L}{2\log2}}^2+L+4.51e^{\frac{L}4}\right\}\\
&<1750
\end{align*}
a contradiction.

\end{enumerate}

\end{proof}

\subsection{Case $m=1$}

Suppose $\Gamma\cap N_0(c_i)\neq\emptyset$. Let $\gamma=\gamma_1=\Gamma\cap N_0(c_i)$ and $\delta=\delta_1=\Gamma\setminus\gamma$.

\begin{thm}
  Suppose $\Gamma$ is the shortest closed geodesic with $\abs{\Gamma\cap\Gamma}\geqslant k\geqslant1750$, $14<L<2\cosh^{-1}(2k+1)$. If $\ell(\delta)<1.44+2\log2$, then there is a pair of pants $\Sigma_0\subseteq\Sigma$ with geodesic boundaries or punctures, that $\Gamma\subseteq\Sigma_0$, and $\Gamma$ is a corkscrew geodesic, i.e. a geodesic in the homotopy class of a curve consisting of the concatenation of a simple arc and another that winds $k$ times along the boundary, see \cite{Y2801}.
\end{thm}

\begin{figure}[htbp]
\centering
\tikzexternaldisable	
\begin{tikzpicture}[declare function={
	R = 42;
	f(\x) = R+.5-sqrt(R*R-\x*\x);
	ratio = 3;
}]
\def\x{2.5}
\draw [thick] (-5.5*\x,{f(5.5*\x)}) arc ({270-asin(5.5*\x/R)}:270:{R} and {R});
\draw [thick] (-5.5*\x,{-f(5.5*\x)}) arc ({90+asin(5.5*\x/R)}:90:{R} and {R});

\draw [thick] (-14.7,2.8) arc (-180:180:{0.52} and {0.12});
\draw [thick] (-14.4,-2.8) arc (-180:180:{0.4} and {0.12});

\draw [thick] (-14.6,2.7) arc (45:-45:{4} and {4});

\foreach \a/\b in {0/black, -11/red} {
	\draw [color=\b, thick] (\a,{-f(\a)}) arc (-90:90:{f(\a)/ratio} and {f(\a)});
	\draw [dashed, color=\b, thick] (\a,{f(\a)}) arc (90:270:{f(\a)/ratio} and {f(\a)});
}

\draw[smooth, color=blue, thick] (-9, {f(9)}) 
	to [start angle=-5, next angle=280] (-8.3,0) 
	to [next angle=5] (-7.8,{-f(7.8)}) coordinate (x);
\draw[smooth, color=blue, thick, dashed] (x) 
	to [start angle=5, next angle=80] (-7.35,0) 
	to [next angle=-5] (-6.7,{f(6.7)}) coordinate (x);
\draw[smooth, color=blue, thick] (x) 
	to [start angle=-5, next angle=280] (-6.1,0) 
	to [next angle=5] (-5.7,{-f(5.7)}) coordinate (x);
\draw[smooth, color=blue, thick, dashed] (x) 
	to [start angle=5, next angle=80] (-5.25,0) 
	to [next angle=-5] (-4.8,{f(4.8)}) coordinate (x);
\draw[smooth, color=blue, thick] (x) 
	to [start angle=-5, next angle=280] (-4.3,0) 
	to [next angle=5] (-4,{-f(4)}) coordinate (x);
\draw[smooth, color=blue, thick, dashed] (x) 
	to [start angle=5, next angle=80] (-3.65,0) 
	to [next angle=-5] (-3.3,{f(3.3)}) coordinate (x);
\draw[smooth, color=blue, thick] (x) 
	to [start angle=-5, next angle=270] (-2.8,0) 
	to [next angle=185] (-3.3,{-f(3.3)}) coordinate (x);
\draw[smooth, color=blue, thick, dashed] (x) 
	to [start angle=185, next angle=100] (-3.65,0) 
	to [next angle=175] (-4,{f(4)}) coordinate (x);
\draw[smooth, color=blue, thick] (x) 
	to [start angle=175, next angle=260] (-4.3,0) 
	to [next angle=185] (-4.8,{-f(4.8)}) coordinate (x);
\draw[smooth, color=blue, thick, dashed] (x) 
	to [start angle=185, next angle=100] (-5.25,0) 
	to [next angle=175] (-5.7,{f(5.7)}) coordinate (x);
\draw[smooth, color=blue, thick] (x) 
	to [start angle=175, next angle=260] (-6.1,0) 
	to [next angle=185] (-6.7,{-f(6.7)}) coordinate (x);
\draw[smooth, color=blue, thick, dashed] (x) 
	to [start angle=185, next angle=100] (-7.35,0) 
	to [next angle=175] (-7.8,{f(7.8)}) coordinate (x);
\draw[smooth, color=blue, thick] (x) 
	to [start angle=175, next angle=260] (-8.3,0)
	to [next angle=185] (-9,{-f(9)}) coordinate (x);
\draw[smooth, color=blue, thick, dashed] (x) 
	to [start angle=185, next angle=105] (-10,0)
	to [next angle=160] (-11,{f(11)}) coordinate (x);
\draw[smooth, color=blue, thick] (x) 
	to [start angle=160, next angle=280] (-13.5,0.5) coordinate (x);
\draw[smooth, color=blue, thick, dashed] (x) 
	to [start angle=320, next angle=340] (-9, {f(9)}) coordinate (x);


  			
\path ({f(0)/ratio},0) node[circle, inner sep=1pt, label={[label distance=.4em, anchor=center]0:$c_i$}]{};

\path (-11, {f(11)}) node[circle, fill, inner sep=1pt, label={[shift={(90:0.45)}, anchor=center]{$P$}}]{};
\path (-11.62, 0.6) node[circle, fill, inner sep=1pt, label={[shift={(315:0.45)}, anchor=center]{$Q$}}]{};
\path (-11.2, 1.3) node[circle, inner sep=1pt, label={[shift={(140:0.45)}, anchor=center]{$\epsilon_1$}}]{};
\path (-10.3, -1.3) node[circle, inner sep=1pt, label={[shift={(140:0.45)}, anchor=center]{$\epsilon_2$}}]{};
\path (-12.5, 1.7) node[circle, inner sep=1pt, label={[shift={(140:0.45)}, anchor=center]{$\delta$}}]{};
\path (-14.3, 2.8) node[circle, inner sep=1pt, label={[shift={(90:0.45)}, anchor=center]{$c'$}}]{};
\path (-14.3, -2.8) node[circle, inner sep=1pt, label={[shift={(270:0.45)}, anchor=center]{$c''$}}]{};

\end{tikzpicture}
\caption{\label{fig:corkscrew}
A corkscrew geodesic: the blue curve in the figure}
\end{figure}
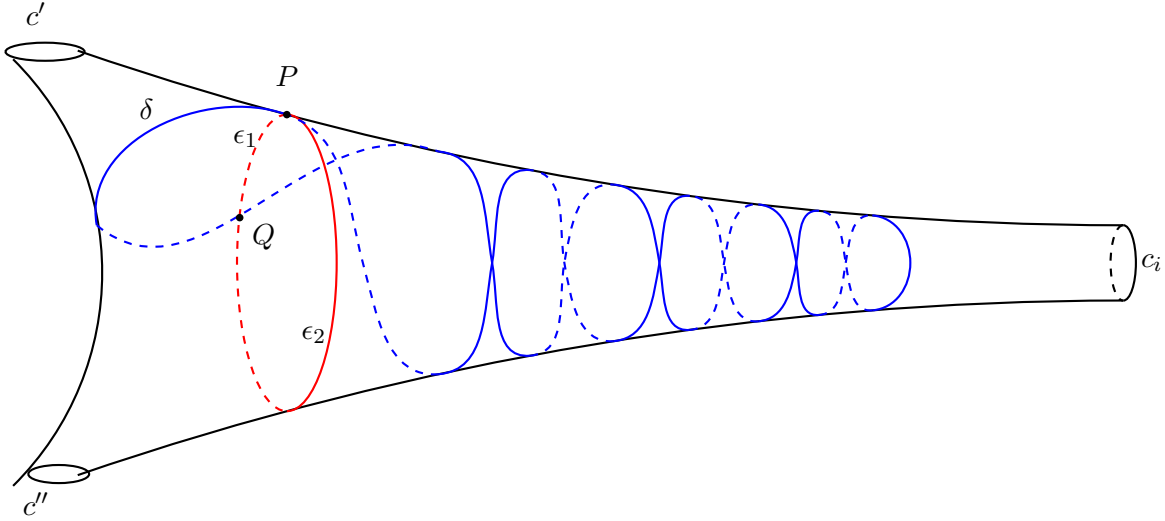

\begin{proof}

  The proof is similar as the proof of \cite[Theorem 1.1(2)]{Y2801}. Let $P,Q\in\Gamma\cap\partial N_0(c_i)$ are endpoints of arc $\gamma$. Then $\gamma$ and $\delta$ both have $P,Q$ as endpoints. Suppose $\epsilon$ is one of two components of $\partial N_0(c_i)$ containing $P,Q$, then $\epsilon$ is a closed curve, is divided to 2 curves $\epsilon_1\cup\epsilon_2$ and $P,Q$ are both endpoints of $\epsilon_1$ and $\epsilon_2$ as in Figure \ref{fig:corkscrew}.

\begin{enumerate}
\item

If $\delta$ has no self-intersection, then both $\delta\cup\epsilon_1$ and $\delta\cup\epsilon_2$ are closed curves with no self-intersection, hence homotopic to closed geodesics or cusps $c'$ and $c''$, and $c_i,c',c''$ is a boundary of a pair of pants $\Sigma_0$. 

We claim that $\Gamma\subseteq\Sigma_0$. If this fails, since clearly $\delta\cup\epsilon_1$ is freely homotopic to a simple closed curve contained in $\Sigma_0$, then if $\delta\cup c'\neq\emptyset$ then $\delta\cup c'$ must create bigons, but both $\delta$ and $c'$ are geodesics, a contradiction. Similarly $\delta\cup c''=\emptyset$. Hence the claim holds.

\item
If $\delta$ has self-intersection, choose a parametrization $\delta:[0,\ell(\delta)]\rightarrow\Sigma$ of $\delta$ and $\delta(0)=P,\delta(\ell(\delta))=Q$. Let $t_2$ be the supremum of all $t$ such that the restriction $\delta|_{[0,t]}$ is a simple arc, then there exists a unique $t_1\in[0,t_2)$ with $\delta(t_1)=\delta(t_2)$. hence $\delta|_{[t_1,t_2]}$ is a simple loop in $\delta$ and denote it by $v$. The simple loop $v$ is noncontractible so we only need to consider whether it is freely homotopic to $c_i$ or not.

\item
If $v$ is freely homotopic to $c_i$, since $v\cap N_0(c_i)=\emptyset$, $v\cup c_i$ is the boundary of an open annulus $A\subseteq\Sigma$. Since $\delta|_{[t_1,t_2]}$ is a closed curve which is geodesic except $\delta(t_1)=\delta(t_2)$, $d(\delta(t),\epsilon)$ first decreases and then increases for $t\in[t_1,t_2]$. There exists $t_3$ such that $t_3$ is the infimum of $t$ such that $\delta(t)\notin A$. Then $\delta(t_3)\in\delta|_{[t_1,t_2]}$ and $t_3\leqslant t_1$. If $t_3<t_1$ then there exists unique $t_4\in[t_1,t_2)$ such that $\delta(t_4)=\delta(t_3)$, contradicts the definition of $t_2$. Hence $t_3=t_1$, and $\delta|_{[0,t_2]}\subseteq\overline{A}$. Hence the function $d_A(\delta(t),\epsilon)$ attains its maximum on $t=t_1$ in $[0,t_2]$, where $d_A$ means the distance function in $\overline{A}$. 

On the other hand, since $\delta|_{[t_1,t_2]}$ is a closed curve which is geodesic except $\delta(t_1)=\delta(t_2)$, and $d_A(\delta(t),\epsilon)$ first decreases and then increases for $t\in[t_1,t_2]$, hence $\frac{d}{dt}d_A(\delta(t),\epsilon)|_{t=t_1}<0$, a contradiction.

\item

If $v$ is not freely homotopic to $c_i$, define closed curve $\chi=\Gamma\setminus v$, if $\chi$ is freely homotopic to a power of $c_i$ (or a horocycle of $c_i$), then $\Gamma$ is freely homotopic to a corkscrew geodesic and hence the theorem holds.

\item
If $v$ is not freely homotopic to $c_i$ and $\chi$ is not freely homotopic to a power of $c_i$, define $x'\in\gamma$ as the unique point satisfying $d(x',\epsilon)=\max_{x\in\gamma}d(x,\epsilon)$. Since $\abs{\delta\cap\delta}\leqslant\frac12\bracket{\frac{1.44+2\log2}{0.48}+1}^2<25$, hence $k_0:=\abs{\gamma\cap\gamma}\geqslant k-25$, suppose all the self-intersection points are $x_1,...,x_{k_0}\in\gamma\cap\gamma$, and $d(x_1,\epsilon)<...<d(x_{k_0},\epsilon)$. For $1\leqslant j\leqslant k_0-1$, there exists geodesic segment $\gamma_j^1,\gamma_j^2\subseteq\gamma$ connecting $x_j,x_{j+1}$, suppose $\gamma_j^0=\gamma_j^1\cup\gamma_j^2$ be a nontrivial closed curve. Clearly 
$$\ell(\gamma_1^0)+...+\ell(\gamma_{k_0-1}^0)\leqslant\ell(\gamma)<L$$
Suppose $r(x_j),r(x')$ are the injective radius of $x_j,x'$, since when $x\in N_0(c_i)$, $r(x)=r'(d(x,\partial N_0(c_i)))$ is a decreasing function on $d(x,\partial N_0(c_i))$, then when $k\geqslant1750$ we have
$$\ell(c_i)\leqslant 2r(x')\leqslant\min_{1\leqslant j\leqslant k_0-1}2r(x_j)\leqslant\min_{1\leqslant j\leqslant k_0-1}\ell(\gamma_j^0)\leqslant\frac{L}{k_0-1}\leqslant\frac{2\cosh^{-1}(2k+1)}{k-50}<0.011$$
Here if $c_i$ is a cusp, $\ell(c_i)=0$. Hence 
$$\ell(\epsilon)=\bracket{\frac{\ell(c_i)}2+\frac2{\ell(c_i)}}\cdot\frac{\ell(c_i)}2<1.01$$

Next we define another shorter closed geodesic $\Gamma''$ with more self-intersection to get a contradiction. Let $\delta_1$ be the shortest orthogonal geodesic from $\epsilon$ to itself, clearly $\delta_1$ has no self-intersection, its endpoints are $X,Y\in\epsilon$. We can choose a curve $\gamma''$ homotopic to $c_i$ of length less than $0.011$ such that $x'\in\gamma''$, define $k_1\gamma''$ is the multicurve of $\gamma''$ of multiplicity $k_1$. Since $\chi$ is not freely homotopic to a multiple of $c_i$ we have $\ell(\delta_1)\leqslant\ell(\delta)-\ell(v)\leqslant\ell(\delta)-2\times0.48$. Define the geodesic $\gamma_0$ with endpoints $P,Q$ in the homotopy class of $\gamma\cup30\gamma''$ (the curve obtained by following $\gamma$ from $Q$ to $x'$, then choose such orientation of $\gamma''$ and winding around $\gamma''$ for 30 times, finally following $\gamma$ from $x'$ to $P$). Clearly the winding number $w(\gamma_0)=w(\gamma)+30$.

Define the geodesic $\gamma_0'\subseteq N_0(c_i)$, with endpoints $X,Y$, such that $\gamma_0'$ winding around $c_i$ with winding number $w(\gamma_0')\in[w(\gamma)+28,w(\gamma)+29)$. Since $w(\gamma_0')<w(\gamma_0)$ we have $\ell(\gamma_0')<\ell(\gamma_0)$. Define $\Gamma'$ is the closed geodesic freely homotopic to the closed curve $\gamma_0'\cup\delta_1$.

 Then $\abs{\Gamma'\cap\Gamma'}\geqslant\abs{\Gamma\cap\Gamma}-25+28>k$, but
$$\ell(\Gamma')\leqslant\ell(\delta_1)+\ell(\gamma_0')\leqslant\ell(\delta)-\ell(v)+\ell(\gamma_0)<\ell(\delta)-0.96+\ell(\gamma)+0.33<\ell(\Gamma)$$
contradicting with the minimality of the length $L$.

\end{enumerate}

\end{proof}

Since the minimal length of all the corkscrew geodesics on pair of pants are computed in \cite{B2803} and \cite{B2795}, we have:

\begin{cor}
 If $L>14$, $\ell(\delta)<1.44+2\log2$, $k\geqslant1750$ then $L\geqslant2\cosh^{-1}(2k+1)$ and the equality holds if $\Gamma$ is a corkscrew geodesic on a thrice-punctured sphere.
\end{cor}

\begin{thm}
 If $L\leqslant14$ or $\ell(\delta)\geqslant1.44+2\log2$, and $k>1750$, then $L\geqslant2\cosh^{-1}(2k+1)$.
\end{thm}

\begin{proof}
 If $L>14$ or $\ell(\delta)\geqslant1.44+2\log2$, as we discussed before, let $l=\ell(\gamma)$, $L-l=\ell(\delta)\geqslant1.44+2\log2$, using Lemma \ref{lemma_32} we have
\begin{align*}
\abs{\Gamma\cap\Gamma}&\leqslant1+2\sinh\frac{l}2+\frac12\bracket{\frac{25}{12}(L-l)+1}^2\\
&\leqslant\max\left\{1+2\sinh\bracket{\frac{L}2-0.72-\log2}+\frac12\bracket{\frac{25}{12}\times(1.44+2\log2)+1}^2, 1+\frac12\bracket{\frac{25}{12}L+1}^2\right\}\\
&\leqslant\max\left\{25+0.244e^{\frac{L}2}, 2.2L^2+2.1L+2\right\}
\end{align*}
The second inequality uses the fact that $1+\sinh\frac{l}2+\frac12\bracket{\frac{25}{12}(L-l)+1}^2$ is a convex function on $l\in[0,L-1.44-2\log2]$. When $k>1750$, then $25+0.244e^{\frac{L}2}>1750$ or $2.2L^2+2.1L+2>1750$, both have $L>17.7$. But when $L>17.7$, $4(25+0.244e^{\frac{L}2})+2<e^{\frac{L}2}$ and $4(2.2L^2+2.1L+2)+2<e^{\frac{L}2}$, contradiction.

If $L\leqslant14$, then we get a contradiction since
$$\abs{\Gamma\cap\Gamma}\leqslant1+2\sinh\frac{L}2+\frac12\bracket{\frac{25}{12}L+1}^2<1750$$

\end{proof}

Hence we finished the proof of Theorem \ref{thm:main}.

\end{document}